\documentclass[11pt,reqno]{amsart}
\usepackage[pagewise]{lineno}
\usepackage{pgfplots}
\pgfplotsset{compat=1.18}
\usetikzlibrary{patterns}

\usepackage{amsmath}
\usepackage{mathrsfs}
\usepackage{amsfonts}
\usepackage{amssymb}
\usepackage[pagewise]{lineno}
\usepackage{url}
\theoremstyle{plain}
\newtheorem{athm}{Theorem}

\newtheorem{theorem}{Theorem}[section]
\newtheorem{proposition}[theorem]{Proposition}
\newtheorem{corollary}[theorem]{Corollary}
\newtheorem{lemma}[theorem]{Lemma}

\theoremstyle{definition}

\newtheorem{example}{Example}[section]
\newtheorem{definition}[theorem]{Definition}

\newtheorem{remark}[theorem]{Remark}

\DeclareMathOperator{\esssup}{ess \ sup}

\DeclareMathOperator{\ho}{H}

\input{tcilatex}

\begin{document}

\title[Stability and limit theorems in RDS]{Stability and limit theorems in random dynamical systems}

\author[Davi Lima]{Davi Lima}
\author[Rafael Lucena]{Rafael Lucena}

\date{\today }

\address[Davi Lima]{Universidade Federal de Alagoas, Instituto de Matemática - UFAL, Av. Lourival Melo Mota, S/N Tabuleiro dos Martins, Maceio - AL, 57072-900, Brasil}
\email{davi.santos@im.ufal.br}

\address[Rafael Lucena]{Universidade Federal de Alagoas, Instituto de Matemática - UFAL, Av. Lourival Melo Mota, S/N
	Tabuleiro dos Martins, Maceio - AL, 57072-900, Brasil}
\email{rafael.lucena@im.ufal.br}
\urladdr{www.im.ufal.br/professor/rafaellucena}

\maketitle

\begin{abstract}
    The robust statistical description of dynamical systems under perturbations is a central problem in ergodic theory. In this paper, we investigate the statistical properties of skew-product maps driven by a subshift of finite type with contracting fiber maps, a setting that naturally encompasses Iterated Function Systems (IFS) and Random Dynamical Systems (RDS). Diverging from the classical perturbative frameworks that rely on the compact embedding of anisotropic Banach spaces, we employ a flexible operator approach based on the Lipschitz regularity of the invariant measure's disintegrations with respect to the Wasserstein metric. Our main results are threefold: first, we prove the quantitative statistical stability of the unique invariant measure under admissible deterministic perturbations, obtaining an explicit modulus of continuity of the form $O(R(\delta) \log \delta)$. Second, we establish the exponential decay of correlations on new pair of spaces of observables. Finally, leveraging this exponential decay and Gordin's method, we prove the Central Limit Theorem for the fluctuations of Birkhoff averages of Lipschitz observables.
\end{abstract}

\section{Introduction}

The qualitative theory of smooth dynamical systems and differential equations is fundamentally concerned with understanding the long-term behavior of orbits and the robustness of their statistical descriptions. A paramount theme in this arena is statistical stability, the study of how invariant measures, particularly equilibrium states and physical measures, respond to small perturbations of the underlying evolution law. 

In recent years, the study of skew-product maps and Random Dynamical Systems (RDS) has garnered significant attention due to their ability to model complex, stochastically driven phenomena across various scientific disciplines. In physics, these systems are essential for describing non-equilibrium statistical mechanics and the transport properties of disordered media. In chemistry, they model reaction kinetics subject to fluctuating environmental parameters. Furthermore, in computer graphics and fractal geometry, Iterated Function Systems (IFS), which naturally embed into the skew-product framework, are the standard mathematical tool for generating complex topological structures and textures. Understanding the exact statistical limits of these systems when subjected to perturbations is therefore not only a deep mathematical question but a broadly applicable one.

Since the foundational works of Alves, Bonatti, and Viana \cite{ABV}, the qualitative continuity of invariant measures has been established for various hyperbolic and partially hyperbolic systems. However, obtaining quantitative statistical stability, explicit rates of convergence for the variation of these measures, remains a challenging frontier. A landmark in this direction is the spectral framework developed by Keller and Liverani \cite{KL}. Their approach establishes that the stability of the transfer operator's leading eigenvector can be controlled by the closeness of the unperturbed and perturbed operators in a suitable "weak" norm, provided they satisfy a uniform Lasota-Yorke inequality. 

While exceptionally powerful, the classical perturbative framework typically operates on a pair of Banach spaces $(B_s, B_w)$ under the crucial topological assumption that the strong space $B_s$ is compactly embedded into the weak space $B_w$. This compactness is the fundamental pillar for invoking the Hennion--Ionescu-Tulcea and Marinescu theorems, which guarantee the quasi-compactness of the transfer operator. Nevertheless, in the presence of complex singularity sets, infinite-dimensional fibers, or non-smooth dynamics, constructing spaces that admit such compact embeddings often constitutes an insurmountable technical barrier.

In this work, we bypass this classical bottleneck. We investigate skew-product maps characterized by a base dynamics modeled by a subshift of finite type and fiber maps that exhibit contraction. Following the novel approach to the thermodynamic formalism introduced in \cite{GLu} and recently expanded in \cite{RRR, RRRSTAB} and \cite{RR}, we evaluate the regularity of the system's invariant measure by analyzing the Lipschitz continuity of its disintegrations. By utilizing the Wasserstein metric to measure the distance between fiber distributions, we construct suitable anisotropic spaces that do not require compact inclusion arguments. Instead, we achieve uniqueness and regularity by directly lifting a known spectral gap from the base system to the transfer operator of the total skew-product.

This paper provides a significant leap forward by utilizing these flexible spectral gap results to yield a complete statistical description of perturbed skew-products. Specifically, our contributions push the boundaries of the existing literature in the following directions:

\begin{itemize}
    \item \textbf{Quantitative Statistical Stability:} We establish that the invariant measure $\mu_\delta$ of the perturbed system converges to the unperturbed measure $\mu_0$ with an explicit, quantitative modulus of continuity. Specifically, for an admissible $R(\delta)$-perturbation, the variation in the $L^\infty$ norm is bounded by $D R(\delta) \log \delta$.
    
    \item \textbf{Exponential Decay of Correlations:} We extend the statistical analysis to a broad class of observables defined on the total space. By carefully projecting these observables onto the base $\sigma$-algebra using the regularity of the disintegrations, we prove that the system exhibits exponential decay of correlations.
    
    \item \textbf{The Central Limit Theorem (CLT):} Most notably, we demonstrate how our non-compact operator approach seamlessly integrates with classical probability theory. By utilizing the exponential decay of correlations to verify the summability of conditional expectations, we apply Gordin's Theorem to prove that the fluctuations of Birkhoff averages for Lipschitz observables satisfy the Central Limit Theorem.
\end{itemize}

The remainder of this paper is organized as follows. In Section 2, we define the setting for the skew-product maps and introduce the necessary functional spaces and transfer operators. Section 3 is devoted to establishing the regularity of the disintegration of the invariant measure using the Wasserstein metric and appropriate spaces of signed measures. In Section 4, we define the admissible perturbations and prove the quantitative statistical stability of the system, as stated in Theorem \ref{d}. Finally, Section 5 presents the proofs of the exponential decay of correlations (Theorems \ref{çljghhjçh} and \ref{athmc}) and the Central Limit Theorem (Theorem \ref{central}).

\textbf{Acknowledgment} This work was partially supported by CNPq (Brazil) Grants 446515/2024-8 and CNPq (Brazil) Grants 420353/2025-9.

\section{Foundational Framework and Notation}\label{sec1}

In this section, we establish the structural assumptions and notation that will be employed throughout this work. For any Lipschitz mapping $\varphi:(X,d_X)\rightarrow (Y,d_Y)$ between two metric spaces, we define its Lipschitz constant as:
\begin{equation}\label{lnot}
    L_{X,Y}(\varphi) := \sup_{x \neq y} \frac{d_Y(\varphi(x), \varphi(y))}{d_X(x, y)}.
\end{equation}
In the specific case where the codomain is the real line ($Y = \mathbb{R}$), we simplify this notation to $L_X(\varphi)$. 

While $L_X(\varphi)$ remains our default convention, we introduce two primary exceptions for the sake of clarity and alignment with the literature. First, when considering real-valued functions defined on sub-shifts of finite type $(\Sigma^{+}_A, d_{\theta})$ (to be detailed below), we represent the Lipschitz constant by $|\varphi|_{\theta} := L_{\Sigma^{+}_A}(\varphi)$. Second, specialized notation will be introduced in Section \ref{lip} to handle the specific requirements of that context.

\subsection{Sub-shifts of Finite Type and Transfer Operators}\label{sec2}

Let $A$ denote an aperiodic matrix. We focus on the associated sub-shift of finite type, represented by the space of admissible sequences:
	$$\Sigma^+_A = \left\{ \underline{x} = (x_i)_{i \in \mathbb{N}} : A_{x_i, x_{i+1}} = 1, \, \forall i \ge 1 \right\},$$
where each symbol $x_i$ belongs to the finite alphabet $\{1, \dots, N\}$. The topology on $\Sigma^+_A$ is induced by the metric $d_\theta$, defined for a fixed parameter $\theta \in (0,1)$ as:
    $$d_\theta(\underline{x}, \underline{y}) = \sum_{i=0}^{\infty} \theta^i (1 - \delta_{x_i, y_i}),$$
where $\delta_{x,y}$ represents the standard Kronecker delta. We fix a Markov measure $m$ on this space, which is invariant under the shift map $\sigma: \Sigma^+_A \to \Sigma^+_A$, where $(\sigma(\underline{x}))_i = x_{i+1}$.

The dynamics of the shift are closely linked to the properties of its Perron-Frobenius operator\footnote{Defined as the unique operator $P_{\sigma}: L_1(m) \to L_1(m)$ satisfying the duality relation $\int \varphi \cdot (\psi \circ \sigma) \, dm = \int P_{\sigma}(\varphi) \cdot \psi \, dm$ for all $\varphi \in L_1(m)$ and $\psi \in L_{\infty}(m)$.}, denoted by $P_\sigma$. We are particularly interested in the action of $P_\sigma$ on the Banach space of Lipschitz functions $\mathcal{F}_{\theta}(\Sigma^+_A) := \{ \varphi: \Sigma^+_A \to \mathbb{R} : |\varphi|_\theta < \infty \}$, equipped with the norm $\|\varphi\|_{\theta} = |\varphi|_\infty + |\varphi|_{\theta}$.

A fundamental property of this operator in our context is the existence of a spectral gap. Specifically, $P_\sigma$ acting on $\mathcal{F}_{\theta}(\Sigma^+_A)$ admits a spectral decomposition (refer to \cite{PB, RE}):
	$$P_{\sigma} = \Pi_{\sigma} + \mathcal{N}_\sigma,$$
where $\Pi_\sigma$ is a projection operator ($\Pi_\sigma^2 = \Pi_\sigma$) and $\mathcal{N}_\sigma$ is a complementary operator with spectral radius $\rho(\mathcal{N}_\sigma) < 1$. Consequently, there exist constants $D > 0$ and $r \in (0,1)$ such that the iterations of the operator on the kernel of $\Pi_\sigma$ satisfy:
	$$\|P_{\sigma}^n \varphi\|_{\theta} \le D r^n \|\varphi\|_{\theta}, \quad \forall n \geq 0.$$

The next result establishes the exponential mixing properties of the shift map, which are a standard implication of the spectral gap of $P_\sigma$ acting on $\mathcal{F}_\theta (\Sigma _A^+)$. This decay of correlations is a well-documented feature in the thermodynamic formalism of sub-shifts (refer to \cite{PB} for a comprehensive treatment); consequently, we state it here without proof.

\begin{proposition}\label{iutryrt}
	One can find a constant $0 < \tau_2 < 1$ such that, for every $\psi \in L^1_{m}(\Sigma_A^+)$ and $\varphi \in \mathcal{F}_\theta(\Sigma_A^+)$, the following correlation estimate holds:
	\[
	\left| 
	\int (\psi \circ \sigma^n) \, \varphi \, dm 
	- \int \psi \, dm \int \varphi \, dm 
	\right| 
	\le \tau_2^n D(\psi, \varphi), \quad \forall\, n \ge 1,
	\]
	where $D(\psi, \varphi) > 0$ is a coefficient determined by the choice of $\psi$ and $\varphi$.
\end{proposition}

\subsection{Dynamics of Contracting Fiber Maps}\label{sec22}

Consider a compact metric space $(K, d)$ endowed with its natural Borel $\sigma$-algebra. We are interested in the transformation $F: \Sigma^+_A \times K \to \Sigma^+_A \times K$, defined as a skew-product of the form:
\begin{equation}\label{cccccc}
F(\underline{x}, z) = (\sigma(\underline{x}), G(\underline{x}, z)),
\end{equation}
where $G: \Sigma^+_A \times K \to K$ represents a measurable fiber map. The product space $\Sigma^+_A \times K$ is assumed to carry the corresponding product $\sigma$-algebra.

For the sake of clarity in our exposition, we occasionally employ the shorthand $\gamma$ to refer to the fiber $\gamma_{\underline{x}}$. However, we maintain a rigorous distinction between these notations within the technical proofs to ensure precision. Furthermore, to streamline our estimates and eliminate redundant multiplicative constants, we adopt the normalization $\text{diam}(K) = 1$. Throughout the remainder of this work, $\pi_1$ and $\pi_2$ shall denote the canonical projections onto the base space $\Sigma^+_A$ and the fiber $K$, respectively.

Finally, we establish a standard notation for measures. Given a measurable space $(X, \mathcal{X})$ and a measure $\mu$ acting upon it, any integrable function $f \in L^1_\mu$ induces a signed measure $f\mu$ on $(X, \mathcal{X})$ through the integration:
$$f\mu(E) := \int_E f \, d\mu, \quad \forall E \in \mathcal{X}.$$

\subsubsection{Assumptions on the Fiber Map $G$}

We now introduce the two core hypotheses regarding the coordinate function $G$, which govern the dynamics along the fibers and their dependence on the base space.

\begin{description}
\item[G1] Let $\mathcal{F}^{s} = \{ \gamma_{\underline{x}} \}_{\underline{x} \in \Sigma_A^+}$ denote the $F$-invariant lamination, where each leaf is defined as $\gamma_{\underline{x}} := \{ \underline{x} \} \times K$. We assume that this lamination is strictly contracting: there exists a contraction ratio $0 < \alpha < 1$ such that, for $m$-almost every $\underline{x} \in \Sigma_A^+$, the following inequality holds for all $y_1, y_2 \in K$:
\begin{equation}
d(G(\underline{x}, y_1), G(\underline{x}, y_2)) \le \alpha d(y_1, y_2). \label{contracting1}
\end{equation}

\item[G2] We assume that the map $G$ exhibits Lipschitz regularity with respect to the base coordinate, uniformly across the fiber. Specifically, for each $y \in K$, there is a constant $k_y \ge 0$ satisfying
$$d(G(\underline{x}^1, y), G(\underline{x}^2, y)) \le k_y d_{\theta}(\underline{x}^1, \underline{x}^2),$$
with the supremum 
\begin{equation}
H := \sup_{y \in K} k_y < \infty \label{sup}    
\end{equation}being finite.
\end{description}

To illustrate the applicability of these conditions, we examine a canonical model based on a contracting Iterated Function System (IFS).

\begin{example}\label{kjfhjsfg}
Consider a collection of pairwise disjoint closed intervals $I_1, \dots, I_d$ within $\mathbb{R}$, and let $I$ be their convex hull. Let $g: \bigcup_{i=1}^d I_i \to I$ be a map such that $|g'| \ge \lambda > 1$, where $g$ maps each $I_i$ diffeomorphically onto $I$. We define the inverse branches $\varphi_i := (g|_{I_i})^{-1}: I \to I_i$. Let $A$ be an all-ones $d \times d$ matrix, and define the fiber map $G: \Sigma_A^+ \times I \to I$ by $G(\underline{x}, y) := \varphi_{x_0}(y)$, where $x_0$ is the first symbol of $\underline{x}$. The resulting skew-product is $F(\underline{x}, y) = (\sigma(\underline{x}), G(\underline{x}, y))$.

The Mean Value Theorem, combined with the expansion of $g$, directly implies that:
$$d(G(\underline{x}, y_1), G(\underline{x}, y_2)) = d(\varphi_{x_0}(y_1), \varphi_{x_0}(y_2)) \le \lambda^{-1} d(y_1, y_2).$$
Thus, \textbf{G1} is satisfied with $K = I$ and $\alpha = \lambda^{-1}$.

To verify \textbf{G2}, consider $\underline{x}^1, \underline{x}^2 \in \Sigma_A^+$. If both sequences lie in the same initial cylinder $[j]$, then $x_0^1 = x_0^2$, which implies $G(\underline{x}^1, y) = G(\underline{x}^2, y)$, making the distance zero. Otherwise, $x_0^1 \neq x_0^2$, and we have:
$$d(G(\underline{x}^1, y), G(\underline{x}^2, y)) = d(\varphi_{x_0^1}(y), \varphi_{x_0^2}(y)) \le \text{diam}(I).$$
Since $x_0^1 \neq x_0^2$ implies $d_{\theta}(\underline{x}^1, \underline{x}^2) \ge 1 > \theta$, we can write:
$$d(G(\underline{x}^1, y), G(\underline{x}^2, y)) \le \frac{\text{diam}(I)}{\theta} d_{\theta}(\underline{x}^1, \underline{x}^2).$$
Setting $R = \text{diam}(I)/\theta$, it follows that $H \le R < \infty$, confirming \textbf{G2}.
\end{example}

As established in the classical work of Hutchinson \cite{Hc}, such systems possess an invariant measure $\mu$ on $K$ characterized by the fixed-point relation:
\begin{equation}\label{nvbbjdjfdf}
\mu = \sum_{i=1}^{N} p_i \varphi_{i}^{\ast} \mu,
\end{equation}
where $(p_1, \dots, p_N)$ is a probability vector. If $m$ represents the Bernoulli measure associated with this vector, then the product measure $m \times \mu$ is $F$-invariant. This specific construction serves as a foundational prototype for the more general results discussed in Section \ref{lp}.

\section{Functional Framework and Measure Disintegration}

\subsection{Foundations of Rokhlin's Disintegration} 

In this section, we review the essential theory of measure disintegration, which allows for the decomposition of a global measure into a family of conditional measures supported on the atoms of a partition.

Consider a probability space $(\Sigma, \mathcal{B}, \mu)$ and let $\Gamma$ be a partition of $\Sigma$ into measurable sets $\gamma \in \mathcal{B}$. We denote by $\pi: \Sigma \to \Gamma$ the canonical projection mapping each point $u \in \Sigma$ to the unique element $\gamma_u \in \Gamma$ that contains it. The quotient space $\Gamma$ is equipped with the $\sigma$-algebra $\widehat{\mathcal{B}}$, defined such that $\mathcal{Q} \subset \Gamma$ is measurable if and only if its preimage $\pi^{-1}(\mathcal{Q})$ belongs to $\mathcal{B}$. This structure induces a quotient measure $\hat{\mu}$ on $(\Gamma, \widehat{\mathcal{B}})$, given by the push-forward $\hat{\mu} = \pi_* \mu$, i.e., $\hat{\mu}(\mathcal{Q}) = \mu(\pi^{-1}(\mathcal{Q}))$.

The following fundamental result establishes the existence of a fiberwise decomposition of $\mu$ (a detailed proof is available in \cite{Kva}, Theorem 5.1.11).

\begin{theorem}(Rokhlin's Disintegration Theorem) \label{rok}
Let $\Sigma$ be a complete and separable metric space, and let $\Gamma$ be a measurable partition\footnote{Recall that a partition $\Gamma$ is measurable if there exists a set of full measure $M_0 \subset \Sigma$ such that, on $M_0$, $\Gamma$ coincides with the limit $\bigvee_{n=1}^{\infty} \Gamma_n$ of an increasing sequence of countable partitions $\Gamma_1 \prec \Gamma_2 \prec \dots \prec \Gamma_n \prec \dots$. Here, $\Gamma_i \prec \Gamma_{i+1}$ indicates that every element of $\Gamma_{i+1}$ is a refinement of some element in $\Gamma_i$.} of $\Sigma$. For any probability measure $\mu$ on $\Sigma$, there exists a disintegration of $\mu$ with respect to $\Gamma$. This consists of a family of probability measures $\{\mu_\gamma\}_{\gamma \in \Gamma}$ on $\Sigma$ and a quotient measure $\hat{\mu} = \pi^* \mu$ satisfying:
		
		\begin{enumerate}
			\item [(a)] The measure $\mu_\gamma$ is concentrated on the atom $\gamma$ (i.e., $\mu_\gamma(\gamma) = 1$) for $\hat{\mu}$-almost every $\gamma \in \Gamma$;
			\item [(b)] For any measurable set $E \in \mathcal{B}$, the mapping $\gamma \longmapsto \mu_\gamma(E)$ is $\widehat{\mathcal{B}}$-measurable; 
			\item [(c)] The global measure of any set $E \in \mathcal{B}$ can be recovered via the integral $\mu(E) = \int_{\Gamma} \mu_\gamma(E) \, d\hat{\mu}(\gamma)$.
		\end{enumerate}
\end{theorem}

The uniqueness of such a decomposition is guaranteed under standard separability conditions (see \cite{Kva}, Proposition 5.1.7).

\begin{lemma} \label{kv}
Assume that the $\sigma$-algebra $\mathcal{B}$ possesses a countable generator. If both $(\{\mu_\gamma\}_{\gamma \in \Gamma}, \hat{\mu})$ and $(\{\mu'_\gamma\}_{\gamma \in \Gamma}, \hat{\mu})$ are disintegrations of $\mu$ relative to the same partition $\Gamma$, then $\mu_\gamma = \mu'_\gamma$ for $\hat{\mu}$-almost every $\gamma \in \Gamma$. 
\end{lemma}

Let $\Sigma = \Sigma_A^+ \times K$ be the product space and denote by $\mathcal{SB}(\Sigma)$ the collection of Borel signed measures acting on $\Sigma$. For any measure $\mu \in \mathcal{SB}(\Sigma)$, we consider its Jordan decomposition $\mu = \mu^+ - \mu^-$, where $\mu^+$ and $\mu^-$ represent the unique positive and negative variations of $\mu$, respectively (refer to Remark \ref{ghtyhh} for further details). Within this framework, we focus on the subspace $\mathcal{AB}_m$, defined as:
\begin{equation}
\mathcal{AB}_m = \left\{ \mu \in \mathcal{SB}(\Sigma) : \pi_1^* \mu^+ \ll m \quad \text{and} \quad \pi_1^* \mu^- \ll m \right\}, \label{thespace1}
\end{equation}
where $m$ is the reference Markov measure and $\pi_1: \Sigma \to \Sigma_A^+$ is the canonical projection onto the base space, defined by $\pi_1(\underline{x}, y) = \underline{x}$.

When dealing with a probability measure $\mu$ within the class $\mathcal{AB}_m$, the application of Rokhlin's Theorem (Theorem \ref{rok}) becomes particularly transparent. For any point $u = (\underline{x}, y) \in \Sigma$, we define the fiber passing through $u$ as $\gamma_u = \{\underline{x}\} \times K$. It follows immediately that if two points $u$ and $\tilde{u}$ share the same projection under $\pi_1$, they belong to the same fiber $\gamma_u = \gamma_{\tilde{u}}$. 

We then consider the collection of all such fibers, $\Gamma = \{ \{\underline{x}\} \times K : \underline{x} \in \Sigma_A^+ \}$, and let $\pi: \Sigma \to \Gamma$ be the associated quotient map. In this setting, $\Gamma$ is naturally identified with the stable lamination $\mathcal{F}^s$ and the quotient map $\pi$ corresponds to the projection $\pi_1$. Furthermore, since the Radon-Nikodym derivative of the quotient measure with respect to $m$ is given by $\phi_1 = \frac{d\hat{\mu}}{dm}$, we can treat $\hat{\mu}$ as the absolutely continuous measure $\phi_1 m$.

To facilitate the analysis of measures within each component of the foliation, we introduce the coordinate projection $\pi_{2,\gamma}: \gamma_{\underline{x}} \to K$, which maps each pair $(\underline{x}, y)$ to its respective $y$-component in $K$. Through this mapping, we can define a collection of measures localized on the reference space $K$.

\begin{definition} 
Let $\mu \in \mathcal{AB}_m$ be a positive measure and $(\{\mu_\gamma\}_\gamma, \hat{\mu} = \phi_1 m)$ its disintegration with respect to the stable lamination $\mathcal{F}^s$. We define the \textbf{restriction of $\mu$ to the fiber $\gamma$} as the positive measure $\mu|_\gamma$ supported on $K$ (as opposed to the leaf $\gamma$ itself), given for each measurable subset $B \subset K$ by:
\begin{equation*}
\mu |_{\gamma}(B) := \pi_{2, \gamma}^* \left( \phi_1(\underline{x}) \mu_\gamma \right)(B), \quad \text{where } \gamma = \gamma_{\underline{x}}. \label{restrictionmeasure}
\end{equation*}
\end{definition}

This construction effectively projects the disintegrated component $\mu_\gamma$, weighted by the base density $\phi_1$, onto the universal fiber $K$. This approach shifts the study of measures on distinct leaves to a single, fixed space, thereby streamlining the application of metrics such as the Wasserstein distance in subsequent sections.

The following result, which establishes the measurability of the fiberwise restriction, was previously demonstrated in \cite{DR}. Consequently, we state it here without its proof.

\begin{lemma}\label{mens}
    Consider $\mu \in \mathcal{AB}_m$. For any measurable subset $B \subset K$, the real-valued mapping $\tilde{c}: \Sigma_A^+ \to \mathbb{R}$ defined by $\tilde{c}(\underline{x}) = \mu |_{\gamma_{\underline{x}}}(B)$ is measurable with respect to the Borel $\sigma$-algebra of $\Sigma_A^+$.
\end{lemma}

To generalize the notion of restriction to signed measures $\mu \in \mathcal{AB}_m$, we employ the Jordan decomposition $\mu = \mu^+ - \mu^-$. The \textbf{restriction of $\mu$ to the fiber $\gamma$} is thus defined as:
\begin{equation}
\mu |_{\gamma} := \mu^+ |_{\gamma} - \mu^- |_{\gamma}.
\end{equation}

\begin{remark}\label{ghtyhh}
    It is important to emphasize that this definition is well-posed and independent of the specific decomposition of $\mu$. As shown in Proposition 9.12 of \cite{DR}, if $\mu$ is expressed as the difference $\mu_1 - \mu_2$ of any two positive measures, then for $m$-almost every $\underline{x} \in \Sigma_A^+$, the equality $\mu|_\gamma = \mu_1|_\gamma - \mu_2|_\gamma = \mu^+|_\gamma - \mu^-|_\gamma$ holds, where $\gamma = \gamma_{\underline{x}}$.
\end{remark}

To quantify the discrepancy between signed measures on $X$, we introduce a variation of the Kantorovich-Wasserstein metric:

\begin{definition} \label{wasserstein}
    For any two signed measures $\mu$ and $\nu$ defined on $X$, the \textbf{Wasserstein-Kantorovich Like} distance is given by the following dual formulation:
    \begin{equation}
    W_{1}^{0}(\mu, \nu) = \sup \left\{ \left| \int g \, d\mu - \int g \, d\nu \right| : L_X(g) \le 1, \, \|g\|_{\infty} \le 1 \right\}.
    \end{equation}
\end{definition}

In the following, we adopt the notation:
\begin{equation}\label{WW}
\|\mu\|_W := W_{1}^{0}(0, \mu).
\end{equation}
It is well known that $\|\cdot\|_W$ constitutes a well-defined norm on the vector space of signed measures over a compact metric space. Notably, this norm is equivalent to the one induced by the dual of the space of Lipschitz functions.

The efficacy of this metric in establishing limit theorems has been explored in several works, including \cite{GP, ben, RRR, RR}. For instance, the authors in \cite{ben} utilize this framework to investigate systems with shrinking fibers, a context closely related to our own.

\begin{remark}
    For the sake of notation, the norm of the fiberwise restriction of $\mu$ is represented as $\|\mu|_\gamma\|_W := W_{1}^{0}(\mu^+|_\gamma, \mu^-|_\gamma)$.
\end{remark}

\subsection{The Space of Lipschitz Measures: $\mathcal{L}_\theta$}\label{lip}

As previously established, any positive measure $\mu$ on the product space $\Sigma_A^+ \times K$ can be decomposed along the stable lamination $\mathcal{F}^s$. This disintegration allows us to interpret $\mu$ as a collection of localized measures $\{\mu|_\gamma\}_{\gamma \in \mathcal{F}^s}$ on the fiber $K$. Given the natural bijection between the leaves in $\mathcal{F}^s$ and the base space $\Sigma_A^+$, this construction effectively defines a mapping:
\begin{equation*}
    \Sigma_A^+ \longrightarrow \mathcal{SB}(K),
\end{equation*}
where the space of signed measures $\mathcal{SB}(K)$ is equipped with the Wasserstein-Kantorovich Like metric introduced in Definition \ref{wasserstein}.

For analytical convenience, we adopt a functional perspective by representing this assignment as a path $\Gamma_\mu: \Sigma_A^+ \to \mathcal{SB}(K)$. Formally, for a given disintegration $(\{\mu_\gamma\}_{\gamma \in \mathcal{F}^s}, \phi_1)$ of $\mu$, the path is defined almost everywhere by the relation $\Gamma_\mu(\gamma) = \mu|_\gamma$.

It is important to note that since the disintegration is uniquely determined only $\hat{\mu}$-almost everywhere, the representative path $\Gamma_\mu$ is not strictly unique. To ensure a well-defined framework, we treat $\Gamma_\mu$ as an equivalence class of paths that coincide almost everywhere, thereby eliminating ambiguity arising from the choice of disintegration.

\begin{definition}\label{defd}
	Let $\mu$ be a Borel signed measure on $\Sigma$ and consider a disintegration $\omega = (\{\mu_{\gamma_{\underline{x}}}\}_{\underline{x} \in \Sigma_A^+}, \phi_1)$, where the family of measures $\{\mu_{\gamma_{\underline{x}}}\}$ is defined $m$-almost everywhere and $\phi_1: \Sigma_A^+ \to \mathbb{R}$ denotes the marginal density. We define the \textbf{class of equivalent paths} associated with $\mu$, denoted by $\Gamma_{\mu}$, as the collection:
	\begin{equation*}
	\Gamma_{\mu} = \{ \Gamma^\omega_{\mu} \}_\omega,
	\end{equation*}
	where $\omega$ varies over all valid disintegrations of $\mu$. For a specific choice of $\omega$, the representative mapping $\Gamma^\omega_{\mu}: \Sigma_A^+ \to \mathcal{SB}(K)$ is given by the fiberwise restriction:
	$$\Gamma^\omega_{\mu}(\underline{x}) = \mu |_{\gamma} = \pi_{2, \gamma}^* \left( \phi_1(\underline{x}) \mu_{\gamma} \right), \quad \text{with } \gamma = \gamma_{\underline{x}}.$$
\end{definition}

For each representative $\Gamma_{\mu}^\omega$, we denote its domain of definition by $I_{\Gamma_{\mu}^\omega} \subset \Sigma_A^+$. To simplify the notation in the following definitions, we fix $\gamma_1 = \gamma_{\underline{x}^1}$ and $\gamma_2 = \gamma_{\underline{x}^2}$ for any pair $(\underline{x}^1, \underline{x}^2) \in \Sigma_A^+ \times \Sigma_A^+$.

\begin{definition} \label{Lips3}
    Consider a specific disintegration $\omega$ of $\mu$ and its corresponding functional representation $\Gamma_{\mu}^\omega$. The \textbf{Lipschitz constant of $\mu$ relative to $\omega$} is defined as:
    \begin{equation}\label{Lips1}
    |\mu|_\theta^\omega := \operatorname{ess\,sup}_{\underline{x}^1, \underline{x}^2 \in I_{\Gamma_{\mu}^\omega}} \left\{ \frac{\|\mu|_{\gamma_1} - \mu|_{\gamma_2}\|_W}{d_\theta(\underline{x}^1, \underline{x}^2)} \right\},
    \end{equation}
    where the essential supremum is taken with respect to the Markov measure $m$. Building upon this, we define the \textbf{Lipschitz constant} of the measure $\mu$ by taking the infimum over all equivalent representations:
    \begin{equation}\label{Lips2}
    |\mu|_\theta := \inf_{\Gamma_{\mu}^\omega \in \Gamma_{\mu}} \{ |\mu|_\theta^\omega \}.
    \end{equation}
\end{definition}

\begin{remark}
    To maintain a concise notation, we shall write $\mu|_\gamma$ instead of $\Gamma_{\mu}^\omega(\underline{x})$ whenever the specific choice of disintegration $\omega$ is clear from the context.
\end{remark}

\begin{definition} \label{erfcscvdsd}
    Building upon Definition \ref{Lips3}, we introduce the space of Lipschitz measures, denoted by $\mathcal{L}_\theta$, which consists of all measures in $\mathcal{AB}_m$ with a finite Lipschitz constant:
    \begin{equation}
    \mathcal{L}_{\theta} = \{ \mu \in \mathcal{AB}_m : |\mu|_\theta < \infty \}.
    \end{equation}
\end{definition}

In the subsequent analysis, we consider the fiber map $F_\gamma: K \to K$ induced by the global dynamics $F$. This map is defined for each leaf $\gamma = \gamma_{\underline{x}}$ as the composition:
\begin{equation}\label{poier}
    F_\gamma(y) = \pi_2 \circ F \circ (\pi_2|_{\gamma})^{-1}(y),
\end{equation}
where $\pi_2: \Sigma_A^+ \times K \to K$ represents the canonical projection onto the second coordinate, $\pi_2(\underline{x}, y) = y$. Furthermore, in the technical proofs that follow, we shall consistently use the notation $\gamma_1 = \gamma_{\underline{x}^1}$ and $\gamma_2 = \gamma_{\underline{x}^2}$ for any points $\underline{x}^1, \underline{x}^2 \in \Sigma_A^+$.

\begin{proposition}\label{ttty}
    The Lipschitz constant of the marginal density $\phi_1$ is bounded by the Lipschitz constant of the measure $\mu$. Specifically, for every $\mu \in \mathcal{L}_\theta$, we have:
    \begin{equation}\label{uit}
        |\phi_{1}|_\theta \leq |\mu|_\theta.
    \end{equation}
\end{proposition}

\begin{proof}
    Consider an arbitrary disintegration $\omega = (\{\mu_\gamma\}_{\gamma \in M}, \phi_1)$ associated with a measure $\mu \in \mathcal{L}_\theta$. By invoking Definition \ref{restrictionmeasure} and the property of the Wasserstein-like norm $\|\cdot\|_W$ established in \eqref{WW}, we observe that:
    \begin{eqnarray*}
        |\phi_{1}(\underline{x}_1) - \phi_{1}(\underline{x}_2)| &=& \left| \mu|_{\gamma_1}(K) - \mu|_{\gamma_2}(K) \right| \\
        &=& \left| \int_K 1 \, d\mu|_{\gamma_1} - \int_K 1 \, d\mu|_{\gamma_2} \right| \\
        &\leq& \|\mu|_{\gamma_1} - \mu|_{\gamma_2}\|_W.
    \end{eqnarray*}
    By normalizing both sides of the inequality by the distance $d_\theta(\underline{x}_1, \underline{x}_2)$ and applying the essential supremum over $\Sigma_A^+$, it follows that:
    $$|\phi_{1}|_\theta \leq |\mu|_\theta^\omega.$$
    Since this bound holds for any choice of disintegration, taking the infimum over the set of all valid $\omega \in \Gamma_\mu$ yields the desired result $|\phi_1|_\theta \leq |\mu|_\theta$, concluding the proof.
\end{proof}

\subsection{The spaces $\mathcal{L}^\infty$ and $S^\infty$}\label{lp}

While the space $\mathcal{L}_\theta$ focuses on the regularity of the paths in the space of measures, it is equally necessary to quantify the uniform boundedness of the fiberwise restrictions. In what follows, we introduce the space $\mathcal{L}^\infty$ of measures with bounded total variation across the fibers, and the space $S^\infty$, which combines this boundedness with the Lipschitz regularity of the marginal density.

\begin{definition}\label{L}
We define the space $\mathcal{L}^{\infty} \subseteq \mathcal{AB}_m(\Sigma)$ as the set of measures whose fiberwise Wasserstein norm is essentially bounded:
\begin{equation}
\mathcal{L}^{\infty} := \left\{ \mu \in \mathcal{AB}_m : \operatorname{ess\,sup}_{\underline{x} \in \Sigma_A^+} \{ \|\mu|_{\gamma_{\underline{x}}}\|_W \} < \infty \right\},
\end{equation}
where the essential supremum is taken over $\Sigma^+_A$ with respect to the Markov measure $m$. This space is equipped with the mapping $\|\cdot\|_\infty: \mathcal{L}^\infty \to \mathbb{R}$, defined as:
\begin{equation}
\|\mu\|_\infty = \operatorname{ess\,sup}_{\underline{x} \in \Sigma_A^+} \{ W_1^0(\mu^+|_{\gamma_{\underline{x}}}, \mu^-|_{\gamma_{\underline{x}}}) \}.
\end{equation}
\end{definition}

Finally, we characterize a refined class of signed measures that integrate both the fiberwise magnitude and the regularity of the base marginal.

\begin{definition}\label{iuytietrsdf}
The space $S^\infty$ is defined as the intersection:
\begin{equation}\label{si}
S^{\infty} = \left\{ \mu \in \mathcal{L}^{\infty} : \phi_1 \in \mathcal{F}_{\theta}(\Sigma^+_A) \right\}, 
\end{equation}
where $\phi_1 = \frac{d(\pi_1^* \mu)}{dm}$ is the Radon-Nikodym derivative of the marginal. We endow $S^\infty$ with the norm:
\begin{equation}
\|\mu\|_{S^{\infty}} = \|\phi_1\|_{\theta} + \|\mu\|_{\infty}.
\end{equation}
\end{definition}

It can be verified that the pairs $(\mathcal{L}^{\infty}, \|\cdot\|_\infty)$ and $(S^{\infty}, \|\cdot\|_{S^\infty})$ constitute well-defined normed vector spaces (see, for instance, \cite{L}). Regarding the measurability of the fiberwise assignments, let $\mathcal{SB}(K)$ be equipped with the Borel $\sigma$-algebra induced by the Wasserstein-Kantorovich Like metric. By Lemma \ref{mens}, the mapping $\tilde{c}: \Sigma_A^+ \to \mathcal{SB}(K)$, defined by $\underline{x} \mapsto \mu|_{\gamma_{\underline{x}}}$, is measurable. This follows from Theorem \ref{rok}, noting that the disintegration map $\underline{x} \mapsto \mu_{\gamma_{\underline{x}}}$ from $\Sigma_A^+$ to $\mathcal{SB}(\Sigma)$, the marginal density $\phi_1: \Sigma_A^+ \to \mathbb{R}$, and the push-forward operator $\pi^*_{2,\gamma}: \mathcal{SB}(\Sigma) \to \mathcal{SB}(K)$ are all measurable functions.

\begin{example} \label{dereroidf}
    Let $\mathcal{C} := \{ \varphi_1, \varphi_2, \dots, \varphi_N \}$ be the finite collection of contractions $\varphi_i: K \to K$ introduced in Example \ref{kjfhjsfg}. We consider the associated Hutchinson invariant measure defined by the relation \eqref{nvbbjdjfdf}. Given that the product measure $m \times \mu$ belongs to the space $S^\infty$, Theorem \ref{probun} implies that:
    \begin{equation}\label{oyiuotyu}
    \mu_0 = m \times \mu.
    \end{equation}
    Building upon the framework established in \cite{GLu} and \cite{RRRSTAB}, we anticipate that this formalism allows for the investigation of statistical stability (specifically, the computation of the modulus of continuity) for the Hutchinson measure $\mu$ when the underlying IFS is subject to deterministic perturbations.
\end{example}

\section{The Transfer Operator and its Functional Properties}\label{eryet}

In this section, we investigate the transfer operator $\func{F^*}$ associated with the skew-product $F$. For any signed Borel measure $\mu \in \mathcal{SB}(\Sigma)$, the action of $\func{F^*}$ is defined by the push-forward:
\begin{equation*}
[F^* \mu](E) = \mu(F^{-1}(E)),
\end{equation*}
for every measurable set $E \subset \Sigma$. Our goal is to characterize how this operator acts on the specialized spaces $\mathcal{L}_\theta$, $\mathcal{L}^\infty$, and $S^\infty$.

To facilitate the upcoming analysis, we define the partition elements $P_i := [0; i] = \{ \underline{x} = (x_m)_{m \in \mathbb{N}} \in \Sigma_A^+ : x_0 = i \}$ and let $\sigma_i := \sigma|_{P_i}$ represent the restriction of the shift map to each cylinder, for $i \in \{1, \dots, N\}$. Furthermore, for any measure $\mu \in \mathcal{AB}_m$, we denote its disintegration along the stable lamination $\mathcal{F}^s$ by the pair $(\{\mu_\gamma\}_\gamma, \phi_1)$, where $\phi_1 = \frac{d(\pi_1^* \mu)}{dm}$ is the corresponding marginal density.

The following lemma establishes the fundamental transformation rules for the marginal density and the fiber measures under the action of $\func{F^*}$. As this result was previously demonstrated in \cite{DR}, we omit its proof here.

\begin{lemma}\label{transformula}
If $\mu \in \mathcal{AB}_m$ is a probability measure with marginal density $\phi_1$, then its image $F^* \mu$ also belongs to $\mathcal{AB}_m$, and its new marginal density satisfies:
\begin{equation}
\frac{d\pi_1^*(\func{F^*} \mu)}{dm} = \mathcal{P}_{\sigma}(\phi_1), \label{1}
\end{equation}
where $\mathcal{P}_{\sigma}$ is the Ruelle-Perron-Frobenius operator associated with the shift. Moreover, let $\underline{x} = (x_n)_{n \in \mathbb{Z}^+} \in \Sigma_A^+$ and denote the fibers by $\gamma = \gamma_{\underline{x}}$ and $\gamma_i = \gamma_{\sigma_i^{-1}(\underline{x})}$. The disintegrated measures of the image, $(\func{F^*} \mu)_\gamma = \nu_\gamma$, are given by:
\begin{equation}
\nu_\gamma = \frac{1}{\mathcal{P}_{\sigma}(\phi_1)(\underline{x})} \sum_{i=1}^N \left( \frac{\phi_1}{J_{m, \sigma_i}} \circ \sigma_i^{-1}(\underline{x}) \right) \cdot \chi_{\sigma_i(P_i)}(\underline{x}) \cdot \func{F^*} \mu_{\gamma_i}, \label{2}
\end{equation}
provided that $\mathcal{P}_{\sigma}(\phi_1)(\underline{x}) \neq 0$. In the case where $\mathcal{P}_{\sigma}(\phi_1)(\underline{x}) = 0$, we define $\nu_\gamma$ as the normalized Lebesgue measure on the fiber $\gamma$. The summation is taken over indices $i$ such that the transition $ix_0$ is admissible ($A_{i,x_0} = 1$). Throughout this work, $\chi_E$ denotes the characteristic function of the set $E$.
\end{lemma}

Recall from Remark \ref{ghtyhh} that the fiberwise restriction $\mu|_\gamma$ is well-defined and independent of the choice of decomposition. Consequently, for any $\mu \in \mathcal{L}^\infty$, we may utilize the Jordan decomposition $F^* \mu = F^*(\mu^+) - F^*(\mu^-)$. By applying Lemma \ref{transformula} to each positive component, we derive a refined expression for the action of the transfer operator on the fibers.

\begin{proposition} \label{niceformulaab}
Let $\gamma = \gamma_{\underline{x}}$ be a stable leaf and consider the fiber map $F_\gamma: K \to K$ defined by the conjugacy:
\begin{equation*}
F_{\gamma} = \pi_2 \circ F|_\gamma \circ \pi_{2,\gamma}^{-1}.
\end{equation*}
Then, for every measure $\mu \in \mathcal{L}^\infty$ and for $m$-almost every $\underline{x} \in \Sigma_A^+$, the restricted measure $(F^* \mu)|_\gamma$ satisfies:
\begin{equation} \label{niceformulaa}
(\func{F^*} \mu)|_\gamma = \sum_{i=1}^N \frac{\func{F_{\gamma_i}^*} (\mu|_{\gamma_i})}{J_{m,\sigma_i}(i\underline{x})} \cdot \chi_{\sigma_i(P_i)}(\underline{x}),
\end{equation}
where $i\underline{x} = \sigma_i^{-1}(\underline{x})$ and $\gamma_i = \gamma_{i\underline{x}}$. As before, the sum is taken over all admissible transitions $(ix_0)$, being zero otherwise.
\end{proposition}

The contractive properties of the transfer operator and the existence of its invariant measure, which we present below, were established in \cite{DR}. Consequently, to maintain the focus on our current functional development, we omit the respective proofs.

\begin{lemma}\label{niceformulaac}
For every measure $\mu \in \mathcal{AB}_m$ and for $m$-almost every stable leaf $\gamma \in \mathcal{F}^s$, the following inequality holds:
\begin{equation}\label{weak1}
\| \func{F_\gamma^*} (\mu|_\gamma) \|_W \leq \| \mu|_\gamma \|_W,
\end{equation}
where $F_\gamma: K \to K$ is the fiber map defined in Proposition \ref{niceformulaab}. Furthermore, if $\mu$ is a probability measure on $K$, then:
\begin{equation}\label{simples}
\| \mu \|_W = 1.
\end{equation}
In particular, since the push-forward $\func{F^*} \mu$ of a probability measure remains a probability measure, we have $\| \func{F^*} \mu \|_W = 1$.
\end{lemma}

\begin{proposition}[Weak Contraction of the Fiberwise Norm] \label{l1}
Suppose that the map $F$ satisfies the contraction hypothesis (G1). If $\mu \in \mathcal{L}^\infty$, then the transfer operator $F^*$ is a weak contraction with respect to the $\|\cdot\|_\infty$ norm:
\begin{equation}\label{weakcontral11234}
\| \func{F^*} \mu \|_\infty \leq \| \mu \|_\infty.
\end{equation}
\end{proposition}

\begin{theorem}\label{probun}
Under the assumption that $F$ satisfies hypothesis (G1), there exists a unique $F$-invariant probability measure $\mu_0 \in S^\infty$. This measure is satisfies the following norm values:
\begin{equation*}
\| \mu_0 \|_\infty = 1 \quad \text{and} \quad \| \mu_0 \|_{S^\infty} = 2.
\end{equation*}
\end{theorem}

In this section, we formalize the asymptotic behavior of the transfer operator. Broadly speaking, we say that $F^*$ exhibits \textbf{convergence to equilibrium} with a decay rate $\Phi$ relative to the norms $\|\cdot\|_{S^{\infty}}$ and $\|\cdot\|_{\infty}$ if, for every measure $\mu$ in the subspace of vanishing marginals $\mathcal{V}$, defined by:
\begin{equation}\label{Vss}
\mathcal{V} = \left\{ \mu \in S^{\infty} : \frac{d\pi_1^* \mu}{dm} \in \ker(\Pi_{\sigma}) \right\},
\end{equation}
the following inequality holds:
\begin{equation}\label{wwe}
\|\func{F^{*n}} \mu \|_{\infty} \leq \Phi(n) \| \mu \|_{S^{\infty}},
\end{equation}
where $\Phi(n) \to 0$ as $n \to \infty$.

\begin{remark}\label{rem123}
    It is important to note that $\mathcal{V}$ naturally encompasses all signed measures with zero total mass (i.e., $\mu(\Sigma) = 0$). To see this, recall that the marginal density is given by $\phi_1 = \frac{d(\pi_1^* \mu)}{dm}$. Thus, for any such measure, the projection $\Pi_\sigma$ satisfies:
    \begin{equation*}
    \Pi_\sigma (\phi_1) = \int_{\Sigma_A^+} \phi_1 \, dm = \int_{\Sigma_A^+} d(\pi_1^* \mu) = \mu(\pi_1^{-1}(\Sigma_A^+)) = \mu(\Sigma) = 0.
    \end{equation*}
\end{remark}

The following result establishes that, under the contraction hypothesis (G1), the convergence rate $\Phi(n)$ is at least exponential. As demonstrated in \cite{DR}, we state the proposition without its proof.

\begin{proposition}[Exponential Convergence to Equilibrium] \label{5.8}
    Assume that the map $F$ satisfies hypothesis (G1). Then, there exist constants $D_2 > 0$ and $0 < \beta_1 < 1$ such that, for every signed measure $\mu \in \mathcal{V}$, the iterates of the transfer operator satisfy:
    \begin{equation*}
    \| \func{F^{*n}} \mu \|_{\infty} \leq D_2 \beta_1^n \| \mu \|_{S^{\infty}},
    \end{equation*}
    for all $n \geq 1$. \label{quasiquasiquasi}
\end{proposition}

The following theorem stands as a cornerstone result of \cite{DR} and constitutes one of the primary analytical engines for the developments presented in this paper. Given its central role in establishing the asymptotic properties of our system, we state it below; however, for the sake of brevity, its proof is omitted here, and the interested reader is directed to \cite{DR} for the complete derivation.

\begin{theorem}[Spectral gap]\label{spgap}
    If the map $F$ satisfies the contraction hypothesis (G1), then the transfer operator $\func{F}^{\ast}: S^\infty \to S^\infty$ admits a decomposition of the form:
    \begin{equation*}
    \func{F}^{\ast} = \func{P} + \func{N},
    \end{equation*}
    where the operators $\func{P}$ and $\func{N}$ satisfy the following properties:
    \begin{enumerate}
        \item[a)] $\func{P}$ is a rank-one projection, i.e., $\func{P}^2 = \func{P}$ and $\dim \operatorname{Im}(\func{P}) = 1$;
        \item[b)] there exist constants $M > 0$ and $0 < \xi < 1$ such that, for every $\mu \in S^\infty$ and all $n \geq 1$:
        \begin{equation*}
        \| \func{N}^n(\mu) \|_{S^\infty} \leq M \xi^n \| \mu \|_{S^\infty};
        \end{equation*}
        \item[c)] the operators satisfy the commutativity and orthogonality relations $\func{P}\func{N} = \func{N}\func{P} = 0$.
    \end{enumerate}
\end{theorem}

The following results provide the essential Lasota-Yorke type inequalities for the transfer operator, establishing how the Lipschitz regularity is controlled under iteration. We alert the reader that the technical proofs for these estimates are omitted here, as they are developed in detail in \cite{DR}.

\begin{proposition}\label{kjsdkjduidf}
Suppose $F$ satisfies hypotheses (G1) and (G2). For any representation $\Gamma_\mu$ of a positive measure $\mu \in \mathcal{L}_\theta$, the following inequality holds:
\begin{equation}
|\Gamma^{\omega}_{\func{F}^*\mu}|_\theta^\omega \leq \theta |\mu|_\theta^\omega + C_1 \|\mu\|_\infty,
\end{equation}
where the constant is given by $C_1 = \max \{ H\theta + \theta N |g|_\theta, 2 \}$.
\end{proposition}

\begin{theorem}\label{hdjfhsdfjsd}
Under the assumptions (G1) and (G2), for every representation $\Gamma_\mu$ of a positive measure $\mu \in \mathcal{L}_\theta$ and for all $n \geq 1$, we have:
\begin{equation}\label{uerjerh}
|\Gamma^{\omega}_{\func{F{^*}}^n\mu}|_\theta^\omega \leq \theta^n |\mu|_\theta^\omega + \frac{C_1}{1 - \theta} \|\mu\|_\infty,
\end{equation}
where $C_1$ is the constant defined in Proposition \ref{kjsdkjduidf}.
\end{theorem}

\begin{remark}\label{kjedhkfjhksjdf}
By taking the infimum over all possible disintegrations $\omega$ on both sides of inequality \eqref{uerjerh}, it follows that for each positive measure $\mu \in \mathcal{L}_\theta$:
\begin{equation}\label{kuyhj}
|\func{F{^*}^n}\mu|_\theta \leq \theta^n |\mu|_\theta + \frac{C_1}{1 - \theta} \|\mu\|_\infty,
\end{equation}
for all $n \geq 1$.
\end{remark}

The third theorem provides an estimate for the Lipschitz constant (see equation (\ref{Lips2}) in Definition \ref{Lips3}) associated with the disintegration of the unique $F$-invariant measure $\mu_0 \in S^{\infty}$. This type of result has numerous applications, and similar estimates for other systems can be found in \cite{RRR}, \cite{BM} and \cite{GLu}. For instance, in \cite{GLu}, the regularity of the disintegration is used to demonstrate the stability of the $F$-invariant measure under a particular type of \textit{ad-hoc} perturbation.

In the following result, we denote by $\mathcal{L}_\theta ^{+}$ the space of positive measures on $\Sigma_A^+ \times K$ for which the Lipschitz constant of their disintegration along $\mathcal{F}^s$, denoted by $|\mu|_\theta$, is finite (see Definitions \ref{Lips3} and \ref{erfcscvdsd}).

\begin{theorem}\label{regg}
Suppose that $F$ satisfies (G1) and (G2). Let $\mu_0$ be the unique $F$-invariant probability measure in $S^{\infty}$. Then $\mu _{0}\in \mathcal{L}_\theta ^{+}$ and 
\begin{equation*}
|\mu _{0}|_\theta \leq \dfrac{C_1}{1-\theta},
\end{equation*}where $C_1>0$ is a constant. 
\end{theorem}

\section{Statistical Stability}\label{loeritu}

In this section, we investigate the response of the invariant measure $\mu_0$ to small variations in the dynamics. We begin by defining the class of perturbations under consideration.
\subsection{Admissible $R(\delta)$-Perturbations}\label{kjrthkje}

We define an \textbf{admissible $R(\delta)$-perturbation} as a family $\{F_{\delta }\}_{\delta \in [0,1)}$, where each map $F_{\delta } = (\sigma_\delta, G_\delta)$ satisfies conditions (G1) and (G2), as well as the following properties (U1), (U2), and (U3).

\begin{enumerate}
   \item[\textbf{(U1)}] There exists $\delta_1 > 0$ small enough such that for all $\delta \in (0, \delta_1)$, the base map is given by $(\sigma_\delta, m_\delta) = (\sigma, m_\delta)$, where $m_\delta$ is the Markov measure defined by the probability vector $p_\delta = (p_{1,\delta}, \dots, p_{N,\delta})$. In particular, the degree remains constant:
    \begin{equation*}
        N := \deg(\sigma) = \deg(\sigma_\delta), \quad \forall \delta \in [0, \delta_1).
    \end{equation*}
    Note that, in this case, for all $\delta \in (0, \delta_1)$, it holds $d_\theta(i\underline{x}_0, i\underline{x}_\delta) = 0$ for all $i=1, \dots, N$, where $i\underline{x}_\delta := \sigma_{\delta,i}^{-1}(\underline{x})$ denotes the $i$-th pre-image of $\underline{x}$ by $\sigma_\delta$. Moreover, it is important to clarify that $i\underline{x}_0 := i\underline{x} = i\underline{x}_\delta$ for all $\delta \in [0, \delta_1)$.

    \item[\textbf{(U2)}] There exists a real-valued function $\delta \mapsto R(\delta) \in \mathbb{R}^+$ such that $$\lim_{\delta \to 0^+} R(\delta) \log \delta = 0,$$ satisfying the following two conditions:
    \begin{enumerate}
        \item[\textbf{(U2.1)}] Let $J_{\sigma_\delta}$ denote the Jacobian of $(\sigma, m_\delta)$ and define the potential $g_\delta(\underline{x}) := 1 / J_{\sigma_\delta}(\underline{x})$. We require that:
        \begin{equation}\label{iueyrtfd}
            \sum_{i=1}^{N} \left| g_\delta(i\underline{x}_\delta) - g_0(i\underline{x}_0) \right| \leq R(\delta), \quad \forall \underline{x} \in \Sigma_A^+,
        \end{equation}where $i\underline{x}_\delta := \sigma_{\delta,i}^{-1}(\underline{x})$ denotes the $i$-th pre-image of $\underline{x}$ by $\sigma_\delta$;
        
        \item[\textbf{(U2.2)}] The fiber maps $G_0$ and $G_\delta$ are $R(\delta)$-close in the supremum norm:
        \begin{equation*}
            d_2(G_0(\underline{x}, y), G_\delta(\underline{x}, y)) \leq R(\delta), \quad \forall (\underline{x}, y) \in \Sigma_A^+ \times K.
        \end{equation*}
    \end{enumerate}

    \item[\textbf{(U3)}] The $\sigma_\delta$-invariant probability measure $m_\delta$ is equivalent to $m$ for all $\delta \in [0, \delta_1)$ ($m \ll m_\delta$ and $m_\delta \ll m$). Moreover, the Radon-Nikodym derivative is uniformly bounded:
    \begin{equation}\label{nmbcvd}
        \func{J} := \sup_{\delta \in [0, \delta_1)} \left\| \frac{dm_\delta}{dm} \right\|_\infty < \infty,
    \end{equation}
    where the $L^\infty$ norm is calculated with respect to $m$.
\end{enumerate}

\begin{remark}\label{toyiout}
	By (U3), since $m \ll m_{\delta}$ for all $\delta$ and $m_{\delta}$ is $\sigma_\delta$-invariant, it follows that $\sum_{i=1}^{\deg(\sigma)}  g_\delta(i\underline{x}_\delta)=1$ $m$-almost everywhere. 
\end{remark}
\begin{enumerate}
	\item [(A1)] (Uniform Lasota-Yorke inequality) There exist constants $B_3>0$ and $0<\beta _3 <1$ such that for all $u \in \mathcal{F}_{\theta}(\Sigma^+_A)$,  all $\delta \in [0,1)$,  and all $n \geq 1$, the following inequality holds:
		
	\begin{equation*}
			||\operatorname{P}_{\sigma_\delta}^n(u)||_{\theta} \leq B_3 \beta _3 ^n || u||_{\theta} + B_3|u|_{\infty},
	\end{equation*}where $||u||_\theta := |u|_\theta + |u|_{\infty}$ and $\operatorname{P}_{\sigma_\delta}$ is the Ruelle-Perron-Frobenius operator of $\sigma_\delta$.
\end{enumerate}

\begin{enumerate}
	\item [(A2)] For all $\delta \in [0,1), $ let $\alpha _\delta$ and $H_\delta$ be the contraction rate $\alpha$ given by Equation (\ref{contracting1}) for $G_\delta$ (see G1) and the constant $H$ defined by Equation (\ref{sup}) (see G2), respectively. Set $$C_{1, \delta} = \max \{ H_\delta \theta +\theta \deg(\sigma) |g_\delta|_\theta , 2\}.$$ Suppose that, $$\sup _ \delta C_{1, \delta} < \infty.$$ 
\end{enumerate}

\begin{remark}
	Let $\{F_{\delta }\}_{\delta \in [0,\delta_1)}$ an admissible $R(\delta)$-perturbation, $i\underline{x}_\delta:=\sigma_{\delta,i}^{-1}(\underline{x})$ and $\gamma _{\delta, i} := \gamma _{\sigma _{\delta, i} ^{-1}(\underline{x})}$ for all $\underline{x} \in \Sigma _A ^+$. Then, for all positive measure $\mu \in \mathcal{L}_\theta$, the following inequality holds: 
    \begin{equation}\label{nslfdflsdjlf}
        \left\vert \left\vert ({\func{F}_{0,\gamma _{0,i} }{_\ast }}- \func{F}_{0,\gamma _{\delta,i} }{_\ast })\mu |_{\gamma _{0,i}}\right\vert \right\vert _{W} =0,
    \end{equation}where $F_{\delta,\gamma _{\delta,i}}$ is defined by Equation (\ref{poier}), for all $\delta \in [0,1)$.
\end{remark}
\begin{lemma}\label{çhjghljk}
	Let $\{F_{\delta }\}_{\delta \in [0,\delta_1)}$ an admissible $R(\delta)$-perturbation, $i\underline{x}_\delta:=\sigma_{\delta,i}^{-1}(\underline{x})$ and $\gamma _{\delta, i} := \gamma _{\sigma _{\delta, i} ^{-1}(\underline{x})}$ for all $\underline{x} \in \Sigma _A ^+$. Then, the following inequality holds: $$\left\vert \left\vert ({\func{F}_{0,\gamma _{\delta,i} }{_\ast }}- \func{F}_{\delta,\gamma _{\delta,i} }{_\ast })\mu |_{\gamma _{0,i}}\right\vert \right\vert _{W} \leq ||\mu|_{\gamma_{0,i}}||_W R(\delta), \forall i=1, \cdots, \deg(\sigma),$$ where $F_{\delta,\gamma _{\delta,i}}$ is defined by Equation (\ref{poier}), for all $\delta \in [0,1)$.
\end{lemma}

\begin{proof}
	To simplify the notation, we denote $\gamma:=\gamma_{\delta,i}$. Thus, by Definition (\ref{wasserstein}) and (U2.2), we have

	\begin{eqnarray*}
		\left\vert \left\vert ({\func{F}_{0,\gamma }{_\ast }}- \func{F}_{\delta,\gamma }{_\ast })\mu |_{\gamma _{0,i}}\right\vert \right\vert _{W} &=& \left\vert \left\vert ({\func{F}_{0, \gamma }{_\ast }}- \func{F}_{\delta, \gamma}{_\ast })\mu |_{\gamma _{0,i}}\right\vert \right\vert _{W} \\&=& \sup _{L_K(u)\leq 1,|u|_{\infty }\leq 1} \left\vert \int {u}  d ({\func{F}_{0, \gamma }{_\ast }}\mu |_{\gamma _{0,i}} - \func{F}_{\delta, \gamma}{_\ast }\mu |_{\gamma _{0,i}}) \right\vert  \\&=& \sup _{L_K(u)\leq 1,|u|_{\infty }\leq 1} \left\vert \int {u (G_0(\gamma,y))  - u (G_\delta(\gamma,y))}  d \mu |_{\gamma _{0,i}} \right\vert \\&\leq & \sup _{L_K(u)\leq 1,|u|_{\infty }\leq 1}  \int { \left\vert u (G_0(\gamma,y))  - u (G_\delta(\gamma,y) ) \right\vert}  d \mu |_{\gamma _{0,i}}  \\&\leq&   \int {d(G_0(\gamma,y),G_\delta(\gamma,y))}  d \mu |_{\gamma _{0,i}}  \\&\leq& R(\delta) \left\vert \int {1}  d \mu |_{\gamma _{0,i}} \right\vert \\&\leq& R(\delta)||\mu |_{\gamma _{0,i}}||_W.
	\end{eqnarray*}
	
\end{proof}

In what follows, for a given admissible $R(\delta)$-perturbation $\{F_\delta \}_{\delta \in [0,\delta_1)}$, we denote by $\func{F_\delta}_{\ast}$ the transfer operator of $F_\delta$. Then, for a given $\delta \in [0,\delta_1)$ we have 

\begin{equation}\label{good}
	(\func{F_\delta}_{\ast}\mu)|_\gamma:=\sum _{i=1}^{\deg(\sigma)}{\func {F}_{\delta, \gamma_i*}\mu|_{\gamma_i}g_\delta(i\underline{x})}\chi_{P_i}(\underline{x}),
\end{equation}where $i\underline{x}_\delta:=\sigma_{\delta,i}^{-1}(\underline{x})$ and $\gamma _{\delta, i} := \gamma _{\sigma _{\delta, i} ^{-1}(\underline{x})}$ for $m$-a.e. $\underline{x} \in \Sigma _A^+$ and all $\mu \in \mathcal{AB}_m$, where $g_\delta(\underline{x})$ is defined in Equation (\ref{iueyrtfd}).

 The next result shows that a statement similar to Proposition \ref{l1} still holds uniformly for $\func{F_\delta}_{\ast}$ for all $\delta$. 

\begin{proposition}\label{agorasim}
	The operator $\func{F_\delta}_{\ast}: \mathcal{L}^{\infty} \longrightarrow \mathcal{L}^{\infty}$ is a weak contraction for all $\delta \in [0,\delta_1)$. It holds, $||\func{F_\delta}_{\ast} \mu ||_\infty \leq ||\mu||_\infty$, for all $\mu \in \mathcal{L}^{\infty}$ and all $\delta \in [0,\delta_1)$.
\end{proposition}
\begin{proof}
	
	The first thing to note is that $\operatorname{P}_{\sigma_\delta} (1) = 1$ $m$-almost everywhere and for all $\delta \in [0,\delta_1)$ (see Remark \ref{toyiout}).

	Now let us complete the proof by applying Lemma \ref{niceformulaac} and Proposition \ref{niceformulaab} in the sequence of inequalities below: for $m$-almost every $\gamma \in \mathcal{F}^s (\approxeq \Sigma _A^+)$, we have (bellow we denote $i\underline{x}_\delta:=\sigma_{\delta,i}^{-1}(\underline{x})$ and $\gamma _{\delta, i} := \gamma _{\sigma _{\delta, i} ^{-1}(\underline{x})}$)
	
	\begin{eqnarray*}
		||(\func{F_\delta}_{\ast} \mu)|_\gamma ||_W 
		&\leq& \sum _{i=1}^{\deg(\sigma)}{|| \func {F_{\delta,\gamma_{\delta,i}}}_{*}\mu|_{\gamma_{\delta, i}}g_\delta(i\underline{x}_\delta)||_W}
		\\ &\leq&  \sum _{i=1}^{\deg(\sigma)}{g_\delta(i\underline{x}_\delta)}||\mu|_{\gamma_{\delta,i}}||_W
		\\ &\leq& ||\mu||_\infty \sum _{i=1}^{\deg(\sigma)}{g_\delta(i\underline{x}_\delta)} 
		\\ &=& ||\mu||_\infty \operatorname{P}_{\sigma_\delta} (1)\\ &=& ||\mu||_\infty.
	\end{eqnarray*} Taking the essential supremum over $\gamma$, with respect to $m$, we have $$||\func{F_\delta}_{\ast} \mu ||_\infty \leq ||\mu||_\infty,$$which is the desired relation.
	
\end{proof}

\begin{lemma}\label{UF2ass}
	Suppose that $\{F_\delta \}_{\delta \in [0,1)}$ is an admissible $R(\delta)$-perturbation, and let $\func{F_\delta}_{\ast}$ and $%
	\mu_{\delta }$ denote the corresponding transfer operators and fixed points (i.e., the $F_\delta$ invariant probability measures in $\mathcal{S}^\infty$), respectively. If the family $\{\mu_{\delta }\}_{\delta \in [0,1)}$ satisfies 
	\begin{equation}\label{new1}
		|\mu_{\delta }|_\theta \leq B_u,
	\end{equation}for all $\delta \in [0,\delta_1)$,
	then there exists a constant $C_{2}$ such that
	\begin{equation}\label{oiuteiyoey}
	||(\func{F_0}_{\ast}-\func{F_\delta}_{\ast})\mu_{\delta }||_{{\infty}}\leq
	C_{1}R(\delta),
	\end{equation}
	for all $\delta \in [0,\delta_1)$, where $C_2:= 2 +B_u$.
\end{lemma}

\begin{proof}

Our approach is to estimate the following quantity:
	
	\begin{equation}\label{12112}
		||(\func{F_0}_{\ast}-\func{F_\delta}_{\ast})\mu_{\delta }||_{{\infty}}= \esssup_{M}{||(\func{F_0}_{\ast}\mu_{\delta})|}_{\gamma }-(\func{F_\delta}_{\ast}\mu_{\delta})|_{\gamma}||_{W}.
	\end{equation}

	Let $\sigma_{\delta,i}$ denote ($1\leq i\leq \deg(\sigma)$) the branches of $\sigma_{\delta}$ defined on the cylinders $P_{i} \in \mathcal{P}$. That is, each branch is given by $\sigma_{\delta,i}=\sigma_{\delta }|_{P_{i}}$. Moreover, recall that we denote $i\underline{x}_\delta:=\sigma_{\delta,i}^{-1}(\underline{x})$ and $\gamma _{\delta, i} := \gamma _{\sigma _{\delta, i} ^{-1}(\underline{x})}$ for all $\underline{x} \in \Sigma _A ^+$. Furthermore, by (U1), there exists $R(\delta)$ such that  
	
	\begin{equation}
		d_\theta(i\underline{x}_0,i\underline{x}_\delta) \leq R(\delta) \ \ \forall i=1 \cdots \deg(f).
	\end{equation}Remember that, by (U1) $\deg (\sigma_\delta) = \deg (\sigma)$ for all $\delta \in [0,\delta_1)$.

Denoting $\func{F}_{\delta,\gamma_{\delta,i}}:=\func{F}_{\delta ,f_{\delta ,i}^{-1}(\gamma )}$ and setting $\mu := \mu_\delta$, we obtain from equation (\ref{good}) that for $\mu _{1}$-almost every $\gamma \in \Sigma_A^+ $, the following holds:

	\[
	(\func{F_0}_{\ast}\mu-\func{F_\delta}_{\ast}\mu )|_{\gamma }=\sum_{i=1}^{\deg(\sigma)}%
	\func{F}_{0,\gamma _{0,i}}{_\ast}\mu |_{\gamma _{0,i}}g_0(i\underline{x}_0)%
	-\sum_{i=1}^{\deg(\sigma)}{\func{F}_{\delta,\gamma _{\delta,i} }{_\ast }}\mu |_{\gamma _{\delta,i}}g_\delta (i\underline{x}_\delta). 
	\]%
	Therefore, we obtain
	
	\begin{equation*}
		||(\func{F_0}_{\ast}\mu-\func{F_\delta}_{\ast})\mu||_{\infty} \leq \func{A}    +   \func{B},
	\end{equation*}where 
	
	\begin{equation}\label{A}
		\func{A} := \esssup_{M}\left\vert \left\vert \sum_{i=1}^{\deg(\sigma)}%
		{\func{F}_{0,\gamma _{0,i} }{_\ast }}\mu |_{\gamma _{0,i}}g_0(i\underline{x}_0)%
		-\sum_{i=1}^{\deg(\sigma)}{\func{F}_{\delta,\gamma _{\delta,i} }{_\ast }}\mu |_{\gamma _{0,i}}g_\delta (i\underline{x}_\delta)\right\vert \right\vert _{W} 
	\end{equation}and 
	
	\begin{equation}	\label{B}
		\func{B} := \esssup_{M}\left\vert \left\vert \sum_{i=1}^{\deg(\sigma)}%
		{\func{F}_{\delta,\gamma _{\delta,i} }{_\ast }}\mu |_{\gamma _{0,i}}g_{\delta}(i\underline{x}_\delta)%
		-\sum_{i=1}^{\deg(\sigma)}%
	{\func{F}_{\delta,\gamma _{\delta,i} }{_\ast }}\mu |_{\gamma _{\delta,i}}g_{\delta}(i\underline{x}_\delta)\right\vert \right\vert _{W}.
	\end{equation}
	
	We now proceed to bound	$\func {A}$ from equation (\ref{A}). By applying the triangle inequality in a similar manner, we obtain
	
	$$ \func{A} \leq \esssup_{M}\func{A}_1(\gamma) + \esssup_{M}\func{A}_2(\gamma),$$ where

	\begin{equation}
		\func{A}_1(\gamma) := \left\vert \left\vert \sum_{i=1}^{\deg(\sigma)}%
		{\func{F}_{0,\gamma _{0,i} }{_\ast }}\mu |_{\gamma _{0,i}}g_{0}(i\underline{x}_0)
		-\sum_{i=1}^{\deg(\sigma)}{\func{F}_{\delta,\gamma _{\delta,i} }{_\ast }}\mu |_{\gamma _{0,i}}g_{0}(i\underline{x}_0)\right\vert \right\vert _{W}
	\end{equation}and

	\begin{equation}
		\func{A}_2(\gamma) := \left\vert \left\vert \sum_{i=1}^{\deg(\sigma)}%
		{\func{F}_{\delta,\gamma _{\delta,i} }{_\ast }}\mu |_{\gamma _{0,i}}g_{0}(i\underline{x}_0)
		-\sum_{i=1}^{\deg(\sigma)}{\func{F}_{\delta,\gamma _{\delta,i} }{_\ast }}\mu |_{\gamma _{0,i}}g_{\delta}(i\underline{x}_\delta)\right\vert \right\vert _{W}.
	\end{equation}Each summand will be analyzed individually. For $\func{A}_1 $, we observe that

	\begin{eqnarray*}
		\func{A}_1(\gamma) &\leq &  \sum_{i=1}^{\deg(\sigma)}%
		\left\vert \left\vert {\func{F}_{0,\gamma _{0,i} }{_\ast }}\mu |_{\gamma _{0,i}}g_{0}(i\underline{x}_0)
		-\sum_{i=1}^{\deg(\sigma)}{\func{F}_{\delta,\gamma _{\delta,i} }{_\ast }}\mu |_{\gamma _{0,i}}g_{0}(i\underline{x}_0)\right\vert \right\vert _{W}
		\\&\leq &  \sum_{i=1}^{\deg(\sigma)}%
		\left\vert \left\vert ({\func{F}_{0,\gamma _{0,i} }{_\ast }}- \func{F}_{\delta,\gamma _{\delta,i} }{_\ast })\mu |_{\gamma _{0,i}}\right\vert \right\vert _{W}g_{0}(i\underline{x}_0)
		\\&\leq &  \sum_{i=1}^{\deg(\sigma)}%
		\left\vert \left\vert ({\func{F}_{0,\gamma _{0,i} }{_\ast }}- \func{F}_{0,\gamma _{\delta,i} }{_\ast })\mu |_{\gamma _{0,i}}\right\vert \right\vert _{W}g_{0}(i\underline{x}_0) \\&+&  \sum_{i=1}^{\deg(\sigma)}%
		\left\vert \left\vert ({\func{F}_{0,\gamma _{\delta,i} }{_\ast }}- \func{F}_{\delta,\gamma _{\delta,i} }{_\ast })\mu |_{\gamma _{0,i}}\right\vert \right\vert _{W}g_{0}(i\underline{x}_0).
	\end{eqnarray*}For $\mu=\mu_{ \delta}$, applying Remark \ref{toyiout}, Equation (\ref{nslfdflsdjlf}), and Lemma \ref{çhjghljk} to the above, we obtain
	
	\begin{eqnarray*}
		\func{A}_1(\gamma) &\leq &  \left(\sum_{i=1}^{\deg(\sigma)}%
		g_{0}(i\underline{x}_0)\right)  R(\delta) ||\mu |_{\gamma _{0,i}}||_W
		\\&\leq & R(\delta). 
	\end{eqnarray*}For $\func{A}_2(\gamma)$, by (U2.1) we have
	\begin{eqnarray*}
		\func{A}_2(\gamma) &\leq&  \sum_{i=1}^{\deg(\sigma)}%
		\left\vert \left\vert {\func{F}_{\delta,\gamma _{\delta,i} }{_\ast }}\mu |_{\gamma _{0,i}}g_{0}(i\underline{x}_0)
		-{\func{F}_{\delta,\gamma _{\delta,i} }{_\ast }}\mu |_{\gamma _{0,i}}g_{\delta}(i\underline{x}_\delta)\right\vert \right\vert _{W}
		\\&\leq& \sum_{i=1}^{\deg(\sigma)}%
		\left\vert g_{0}(i\underline{x}_0)
		-g_{\delta}(i\underline{x}_\delta)\right\vert  \left\vert \left\vert {\func{F}_{\delta,\gamma _{\delta,i} }{_\ast }}\mu |_{\gamma _{0,i}}\right\vert \right \vert _{W}
		\\&\leq& \sum_{i=1}^{\deg(\sigma)}%
		\left\vert g_{0}(i\underline{x}_0)
		-g_{\delta}(\gamma _{\delta,i})\right\vert
		\\&\leq&  R(\delta).
	\end{eqnarray*}We now estimate $\func{B}$. By Remark \ref{toyiout}, we observe that $\sum_{i=1}^{\deg(\sigma)} \left\vert g_{\delta}(i\underline{x}_\delta)\right\vert =1$ $m$-almost every point. Consequently, we obtain
		\begin{eqnarray*}
		\func{B} &\leq&  \esssup_{\Sigma_A^+} \sum_{i=1}^{\deg(\sigma)}%
		\left\vert \left\vert {\func{F}_{\delta,\gamma _{\delta,i} }{_\ast }}\mu |_{\gamma _{0,i}}g_{\delta}(i\underline{x}_\delta)
		-{\func{F}_{\delta,\gamma _{\delta,i} }{_\ast }}\mu |_{\gamma _{\delta,i}}g_{\delta}(i\underline{x}_\delta)\right\vert \right\vert _{W}
		\\&\leq&  \esssup_{\Sigma_A^+} \sum_{i=1}^{\deg(\sigma)} \left\vert g_{\delta}(i\underline{x}_\delta)\right\vert
		\left\vert \left\vert {\func{F}_{\delta,\gamma _{\delta,i} }{_\ast }}(\mu |_{\gamma _{0,i}}-\mu |_{\gamma _{\delta,i}})\right\vert \right\vert _{W}
		\\&\leq&  \esssup_{\Sigma_A^+} \sum_{i=1}^{\deg(\sigma)} \left\vert g_{\delta}(i\underline{x}_\delta)\right\vert
		\left\vert \left\vert \mu |_{\gamma _{0,i}}-\mu |_{\gamma _{\delta,i}}\right\vert \right\vert _{W}
		\\&\leq&  \esssup_{\Sigma_A^+} \sum_{i=1}^{\deg(\sigma)} \left\vert g_{\delta}(i\underline{x}_\delta)\right\vert
		d_\theta(i\underline{x}_\delta,i\underline{x}_0)|\mu|_\theta
		\\&\leq& \esssup_{\Sigma_A^+} \sum_{i=1}^{\deg(\sigma)} \left\vert g_{\delta}(i\underline{x}_\delta)\right\vert
		R(\delta)  |\mu|_\theta
		\\&\leq&  R(\delta)  B_u.
	\end{eqnarray*}Combining these observations, we obtain
	\begin{eqnarray*}
		||(\func{F_0{_\ast }}-\func{F_\delta{_\ast }})\mu_{\delta }||_{\infty} & \leq & \func{A}  +   \func{B}
		\\& \leq & \esssup_{M} \func{A}_1(\gamma)    + \esssup_{M} \func{A}_2(\gamma) +   \func{B}
		\\& \leq & R(\delta) +  R(\delta) + R(\delta)  B_u
		\\& \leq &C_2 R(\delta),
	\end{eqnarray*}where $C_2:=B_u +2.$
\end{proof}

The following result is a key tool in proving Theorem \ref{d}. It establishes that the function $$\delta \longmapsto |\mu_{\delta }|_\theta$$(see Definition \ref{Lips3}) is uniformly bounded, where  $\{\mu_\delta\}_{\delta \in [0,1)}$ denotes the family of $F_{\delta }$-invariant probability measures associated with an admissible perturbation $\{F_{\delta }\}_{\delta \in [0,1)}$ of $F(=F_0)$. 

We begin with a preliminary lemma.

\begin{lemma}	Given an admissible $R(\delta)$-perturbation $\{F_{\delta }\}_{\delta \in [0,1)}$, there exists a constant $C_{1,u}>0$ such that for every positive measure $\mu \in \mathcal{L}_\theta$ the following inequality holds:
	\begin{equation}\label{er}
		|\Gamma^{\omega}_{\func{F{^*}}^n\mu}|_\theta ^\omega \leq \theta ^n |\Gamma _\mu^\omega|_\theta + \dfrac{C_{1,u}}{1-\theta}||\mu||_\infty,
	\end{equation}for all $\delta \in [0,1)$ and all $n \geq 0$.
	\label{las123rtryrdfd2}
\end{lemma}
\begin{proof}
By applying Theorem \ref{hdjfhsdfjsd} to each mapping $F_\delta$, we obtain
	\begin{equation*}
	|\Gamma^{\omega}_{\func{F{^*}}^n\mu}|_\theta ^\omega \leq \theta^n |\mu|_\theta ^\omega + \dfrac{C_{1,\delta}}{1- \theta} ||\mu||_\infty,
	\end{equation*}where $C_{1,\delta}$ was defined in Proposition \ref{kjsdkjduidf} by $C_{1, \delta} = \max \{ H_\delta \theta +\theta \deg(\sigma) |g_\delta|_\theta , 2\}$. Using assumption (A2), we then define $C_{1,u}:= \displaystyle{\sup_\delta C_{1,\delta}},$ which completes the proof.
	
\end{proof}

\begin{remark}\label{riirorpdf}
    Fix a probability measure $\nu$ on $K$ and define the product measure $m_1 = m \times \nu$, where $m$ is the Markov measure established in Subsection \ref{sec1}. We consider the trivial disintegration $\omega_0 = (\{m_{1, \gamma}\}_{\gamma}, \phi_1)$, where the fiber measures are given by the push-forward $m_{1, \gamma} = \func{\pi_{2, \gamma}^{-1 *}} \nu$ for all $\gamma$, and the marginal density is $\phi_1 \equiv 1$. Under this construction, it follows that the fiberwise restriction satisfies:
    \begin{equation*}
    m_1|_\gamma = \nu, \quad \forall \, \gamma.
    \end{equation*}
    Consequently, the associated path $\Gamma^{\omega_0}_{m_1}$ is constant, i.e., $\Gamma^{\omega_0}_{m_1}(\gamma) = \nu$ for all $\gamma \in \mathcal{F}^s$, which directly implies:
    \begin{equation}\label{oiyiye}
    |m_1|_\theta^{\omega_0} = 0.
    \end{equation}
    For each $n \in \mathbb{N}$, let $\omega_n$ be the disintegration of the iterated measure $\func{F}^{*n} m_1$ induced by $\omega_0$ through the recursive application of Lemma \ref{transformula}.

   To describe the path $\Gamma^{\omega_n}_{\func{F}_\delta^{*n} m_1}$ associated with a fixed parameter $\delta \in [0, 1)$, we recall some combinatorial notation. A finite string $\underline{a} = (a_0, a_1, \dots, a_{n-1}) \in \{1, 2, \dots, N\}^n$ is said to be \textit{allowed} if the transitions are admissible, i.e., $A_{a_{i-1}, a_i} = 1$ for all $i = 1, \dots, n-1$. We denote the set of all such allowed strings of length $n$ by $\mathcal{A}_n$. 

For any $\underline{a} \in \mathcal{A}_n$, the inverse branch $\sigma_{\underline{a}}^{-n}$ is well-defined on the set $\{ \underline{x} \in \Sigma_A^+ : A_{a_{n-1}, x_0} = 1 \}$, mapping into the cylinder $[0; \underline{a}] := \{ \underline{y} \in \Sigma_A^+ : y_0 = a_0, \dots, y_{n-1} = a_{n-1} \}$. This map is explicitly given by:
\begin{equation*}
\sigma^{-n}_{\underline{a}}(\underline{x}) = \sigma^{-1}_{a_0} \circ \sigma^{-1}_{a_1} \circ \dots \circ \sigma^{-1}_{a_{n-1}}(\underline{x}) = \underline{a}\underline{x}.
\end{equation*}

Following the construction in Proposition \ref{niceformulaab}, let us define the iterated fiber operator for the map $\func{F}_\delta$ and the corresponding Jacobian:
\begin{equation*}
\func{F}_{\delta, \gamma_{\underline{a}}}^{*n} := \func{F}_{\delta, \gamma_{a_0}}^* \circ \func{F}_{\delta, \gamma_{a_1}}^* \circ \dots \circ \func{F}_{\delta, \gamma_{a_{n-1}}}^*
\end{equation*}
and
\begin{equation*}
J_{m, \underline{a}}(\sigma_{\underline{a}}^{-n}(\underline{x})) := J_{m, \sigma_{a_0}}(\sigma^{-1}_{a_0}(\underline{x})) \cdot J_{m, \sigma_{a_1}}(\sigma^{-2}_{a_0 a_1}(\underline{x})) \cdots J_{m, \sigma_{a_{n-1}}}(\sigma_{\underline{a}}^{-n}(\underline{x})).
\end{equation*}
Then, for $m$-almost every $\underline{x} = (x_j)_{j \in \mathbb{Z}^+}$, the path of the $n$-th iterate under the perturbed operator $\func{F}_\delta^*$ satisfies:
\begin{equation}\label{oiyiy}
\Gamma^{\omega_n}_{\func{F}_\delta^{*n} m_1}(\underline{x}) = \sum_{\underline{a}x_0 \in \mathcal{A}_{n+1}} \frac{\func{F}_{\delta, \gamma_{\underline{a}}}^{*n} \nu}{J_{m, \underline{a}}(\sigma_{\underline{a}}^{-n}(\underline{x}))}.
\end{equation}
By applying Theorem \ref{hdjfhsdfjsd} to the initial condition established in \eqref{oiyiye}, we conclude that the Lipschitz constant of the iterated path for the system $F_\delta$ remains uniformly bounded:
\begin{equation}\label{nbmjsdfjf}
|\Gamma^{\omega_n}_{\func{F}_\delta^{*n} m_1}|_\theta^\omega \leq \frac{C_1}{1 - \theta}, \quad \forall \, n \geq 1.
\end{equation}
\end{remark}

The following lemma is central to our stability analysis, as it establishes that the family of invariant measures inherits a uniform spatial regularity that is independent of the perturbation parameter $\delta$. This uniform bound is the key ingredient for proving the statistical stability of the system.

\begin{lemma} \label{thshgf}
    Consider an admissible $R(\delta)$-perturbation $\{F_\delta\}_{\delta \in [0,1)}$ and let $\{\mu_\delta\}_{\delta \in [0,1)}$ be the corresponding family of unique $F_\delta$-invariant probability measures in $S^\infty$. Then, there exists a uniform constant $B_u > 0$ such that:
    \begin{equation*}
        |\mu_\delta|_\theta \leq B_u,
    \end{equation*}
    for all $\delta \in [0,1)$. Consequently, the family $\{\mu_\delta\}_{\delta \in [0,1)}$ satisfies Equation \eqref{new1}, and the perturbation $\{F_\delta\}_{\delta \in [0,1)}$ satisfies Equation \eqref{oiuteiyoey}.
\end{lemma}

\begin{proof}
    By Theorem \ref{probun}, for each $\delta \in [0,1)$, there exists a unique $F_\delta$-invariant probability measure $\mu_\delta \in S^\infty$. Following the framework of Definition \ref{defd}, we fix a representative path $\Gamma^\omega_{\mu_\delta}|_{\widehat{M_\delta}}$ within the equivalence class $\Gamma_{\mu_\delta}$. Now, consider the reference measure $m_1$ and its iterates $\func{F}_\delta^{*n} m_1$ as defined in Remark \ref{riirorpdf}. The application of Remark \ref{riirorpdf} to these iterates generates a sequence of representations $\{\Gamma^{\omega_n}_{\func{F}_\delta^{*n} m_1}\}_{n \in \mathbb{N}}$. For brevity, we denote:
    \begin{equation*}
    \Gamma_\delta^n := \Gamma_{\func{F}_\delta^{*n} m_1}^{\omega_n}|_{\widehat{M_\delta}} \quad \text{and} \quad \Gamma_\delta := \Gamma^\omega_{\mu_\delta}|_{\widehat{M_\delta}}.
    \end{equation*}
    
    According to the Spectral Gap Theorem (Theorem \ref{spgap}), the sequence of measures $\{\func{F}_\delta^{*n} m_1\}_{n \in \mathbb{N}}$ converges to $\mu_\delta$ in the $L^\infty$ sense. This convergence implies that the sequence of fiber paths $\{\Gamma_\delta^n\}_{n \in \mathbb{N}}$ converges $m$-a.e. to $\Gamma_\delta$ in the space $\mathcal{SB}(K)$ with respect to the Wasserstein metric. In particular, we have pointwise convergence $\Gamma_\delta^n(x) \to \Gamma_\delta(x)$ on a subset of full measure $\widehat{M_\delta} \subset \Sigma_A^+$.
    
    We now show that the Lipschitz regularity is preserved in the limit. For any pair of points $x, y \in \widehat{M_\delta}$, the continuity of the norm implies:
    \begin{equation*}
        \lim_{n \to \infty} \frac{\| \Gamma_\delta^n(x) - \Gamma_\delta^n(y) \|_W}{d_\theta(x, y)} = \frac{\| \Gamma_\delta(x) - \Gamma_\delta(y) \|_W}{d_\theta(x, y)}.
    \end{equation*}
    By Theorem \ref{hdjfhsdfjsd} (as invoked in Remark \ref{riirorpdf}), the left-hand side is uniformly bounded: $|\Gamma_\delta^n|_\theta \leq \frac{C_{1,u}}{1-\theta}$ for all $n \geq 1$, where $C_{1,u}$ is a constant independent of $\delta$. It follows that:
    \begin{equation*}
        \frac{\| \Gamma_\delta(x) - \Gamma_\delta(y) \|_W}{d_\theta(x, y)} \leq \frac{C_{1,u}}{1-\theta},
    \end{equation*}
    which yields the estimate $|\Gamma^\omega_{\mu_\delta}|_\theta \leq \frac{C_{1,u}}{1-\theta}$. Taking the infimum over all valid representations $\omega \in \Gamma_{\mu_\delta}$, we obtain $|\mu_\delta|_\theta \leq B_u$, where $B_u := \frac{C_{1,u}}{1-\theta}$. This completes the proof of the first assertion. The second part of the lemma follows immediately from Lemma \ref{UF2ass}.
\end{proof}

\subsection{Perturbation Theory for Operators}\label{perturbationoperators}

In this section, we provide the abstract framework necessary to establish the statistical stability of our system. By viewing the transfer operators as elements of a family of perturbed operators, we can derive quantitative estimates for the variation of their fixed points.

\begin{definition}
    Let $(B_{w}, \|\cdot\|_{w})$ and $(B_{s}, \|\cdot\|_{s})$ be two normed vector spaces such that $B_{s} \subset B_{w}$ and $\|\cdot\|_{s} \geq \|\cdot\|_{w}$. Suppose that $\func{T}_{\delta}: B_{w} \to B_{w}$ and $\func{T}_{\delta}: B_{s} \to B_{s}$ are well-defined operators for each $\delta \in [0, 1)$, and let $\mu_{\delta} \in B_{s}$ denote a fixed point for $\func{T}_{\delta}$. We say that $\{\func{T}_{\delta}\}_{\delta \in [0, 1)}$ is a \textbf{$R(\delta)$-family of operators} if the following conditions are satisfied:
    
    \begin{enumerate}
        \item[(O1)] There exist a constant $C > 0$ and a function $\delta \mapsto R(\delta) \in \mathbb{R}^+$ such that $\lim_{\delta \to 0^+} R(\delta) \log(\delta) = 0$, and the difference between the unperturbed and perturbed operators satisfies:
        \begin{equation*}
            \|(\func{T}_{0} - \func{T}_{\delta})\mu_{\delta}\|_{w} \leq R(\delta) C, \quad \forall \delta \in [0, 1);
        \end{equation*}
        
        \item[(O2)] The fixed points are uniformly bounded in the strong norm, i.e., there exists $\func{Y} > 0$ such that:
        \begin{equation*}
            \|\mu_{\delta}\|_{s} \leq \func{Y}, \quad \forall \delta \in [0, 1);
        \end{equation*}
        
        \item[(O3)] The unperturbed operator $\func{T}_0$ exhibits a spectral gap in the sense that there exist constants $0 < \rho_2 < 1$ and $C_2 > 0$ such that:
        \begin{equation*}
            \|\func{T}^{n}_0 \mu\|_{w} \leq \rho_2^n C_2 \|\mu\|_{s}, \quad \forall \mu \in \mathcal{V}_s := \{ \mu \in B_{s} : \mu(\Sigma) = 0 \};
        \end{equation*}
        
        \item[(O4)] The family is uniformly power-bounded in the weak norm: there exists $M_2 > 0$ such that for all $\delta \in [0, 1)$, $n \in \mathbb{N}$, and $\nu \in B_{s}$:
        \begin{equation*}
            \|\func{T}_{\delta}^n \nu\|_{w} \leq M_2 \|\nu\|_{w}.
        \end{equation*}
    \end{enumerate}
    The spaces $(B_{w}, \|\cdot\|_{w})$ and $(B_{s}, \|\cdot\|_{s})$ are referred to as the \textbf{weak} and \textbf{strong} spaces of the family, respectively.
\end{definition}

The following lemma establishes a quantitative relationship between the variation of the fixed points $\{\mu_\delta\}_{\delta \in [0, 1)}$ and the perturbation parameter $\delta$. Specifically, it demonstrates that the mapping $\delta \mapsto \mu_{\delta}$ is continuous at $\delta = 0$ with respect to the weak norm and provides an explicit modulus of continuity. For a detailed proof of this result, we refer the reader to \cite{RRRSTAB}, noting that the estimate presented here corresponds to the specific case where $\zeta = 1$.

\begin{lemma}[Quantitative Stability for Fixed Points] \label{dlogd}
    Suppose $\{\func{T}_{\delta}\}_{\delta \in [0, 1)}$ is a $R(\delta)$-family of operators, where $\mu_{0}$ is the unique fixed point of $\func{T}_{0}$ in $B_{w}$ and $\mu_{\delta}$ is a fixed point of $\func{T}_{\delta}$. Then, there exist constants $D_1 < 0$ and $\delta_0 \in (0, 1)$ such that for all $\delta \in [0, \delta_0)$, it holds:
    \begin{equation*}
        \|\mu_{\delta} - \mu_{0}\|_{w} \leq D_1 R(\delta) \log \delta.
    \end{equation*}
\end{lemma}

The following lemma bridges our specific dynamical system with the abstract perturbation framework developed in the previous section.

\begin{lemma}\label{rrr}
    Let $\{F_\delta\}_{\delta \in [0,1)}$ be an admissible $R(\delta)$-perturbation and let $\{\func{F}^*_\delta\}_{\delta \in [0,1)}$ be the induced family of transfer operators. Then, $\{\func{F}^*_\delta\}_{\delta \in [0,1)}$ constitutes an $R(\delta)$-family of operators, where $(\mathcal{L}^{\infty}, \|\cdot\|_\infty)$ and $(S^\infty, \|\cdot\|_{S^\infty})$ serve as the weak and strong spaces of the family, respectively.
\end{lemma}

\begin{proof}
    We must verify that the family $\{\func{F}^*_\delta\}_{\delta \in [0,1)}$ satisfies conditions (O1)--(O4). To establish (O2), we recall that $\mu_\delta$ is an $F_\delta$-invariant probability measure. By combining Proposition \ref{ttty}, Lemma \ref{thshgf}, and the fact that $\|\mu_\delta\|_\infty = |\phi_{1,\delta}|_\infty$ (since $\pi_1^* \mu_\delta = m$), we observe that the strong norm of the invariant measure satisfies:
    \begin{eqnarray*}
        \|\func{F}_\delta^{*n} \mu_\delta\|_{S^\infty} &=& |\mathcal{P}_{\sigma_\delta}^n(\phi_{1,\delta})|_\theta + \|\func{F}_\delta^{*n} \mu_\delta\|_\infty \\
        &\leq& B_3 \beta_3^n |\phi_{1,\delta}|_\theta + B_3 |\phi_{1,\delta}|_\infty + \|\mu_\delta\|_\infty \\
        &\leq& B_3 \beta_3^n |\mu_\delta|_\theta + (B_3 + 1) \func{J} \\
        &\leq& B_3 B_u + (B_3 + 1) \func{J},
    \end{eqnarray*}
    where we used the uniform bound $B_u$ from Lemma \ref{thshgf} and the bound $\func{J}$ for the marginal densities. Thus, (O2) holds with the constant $\func{Y} = B_3 B_u + (B_3 + 1) \func{J}$.

    Condition (O1) follows directly from the definition of an admissible perturbation and the uniform spatial regularity established in Lemma \ref{thshgf}. Finally, properties (O3) and (O4) are consequences of the exponential convergence to equilibrium (Proposition \ref{5.8}) and the uniform power-boundedness (Proposition \ref{agorasim}) applied to each perturbed map $F_\delta$, respectively.
\end{proof}

We are now in a position to state the main result of this paper, which provides a quantitative bound on the statistical stability of the system under deterministic perturbations.

\begin{athm}[Quantitative Stability for Deterministic Perturbations] \label{d}
    Let $\{F_\delta\}_{\delta \in [0,1)}$ be an admissible $R(\delta)$-perturbation, and let $\mu_0$ and $\mu_\delta$ be the unique invariant probability measures in $S^\infty$ for the maps $F$ and $F_\delta$, respectively. Then, there exist constants $D_2 < 0$ and $\delta_1 \in (0, \delta_0)$ such that, for all $\delta \in [0, \delta_1)$:
    \begin{equation}\label{stabll}
        \|\mu_\delta - \mu_0\|_\infty \leq D_2 R(\delta) \log \delta.
    \end{equation}
\end{athm}

\begin{proof}
    The result follows by applying the abstract stability Lemma \ref{dlogd} to the $R(\delta)$-family of transfer operators $\{\func{F}^*_\delta\}_{\delta \in [0,1)}$ identified in Lemma \ref{rrr}.
\end{proof}

\begin{remark}
    By definition, the weak and strong norms satisfy the inequality $\|\cdot\|_W \leq \|\cdot\|_\infty$. Consequently, assuming that the family $\{F_\delta\}_{\delta \in [0,1)}$ satisfies the conditions of Theorem \ref{d}, there exists a constant $D_2 < 0$ such that:
    \begin{equation*}
        \|\mu_\delta - \mu_0\|_W \leq D_2 R(\delta) \log \delta.
    \end{equation*}
    This estimate has immediate implications for physical observables. For any Lipschitz function $g: \Sigma \to \mathbb{R}$, we obtain the following bound on the variation of the expectation values:
    \begin{equation*}
        \left| \int g \, d\mu_\delta - \int g \, d\mu_0 \right| \leq D \|g\|_\theta R(\delta) \log \delta,
    \end{equation*}
    where $\|g\|_\theta = |g|_\infty + L_\Sigma(g)$ is the standard Lipschitz norm (see Equation \eqref{lnot}). This shows that for any such observable, the convergence $\int g \, d\mu_\delta \to \int g \, d\mu_0$ as $\delta \to 0$ occurs with a rate at least of order $R(\delta) \log \delta$.
\end{remark}

Several relevant classes of perturbations of $F$ allow for a linear choice of $R(\delta)$. In cases where $R(\delta)$ is of the form $R(\delta) = K_5 \delta$ for some constant $K_5 > 0$, we obtain the following specialized result.

\begin{corollary}[Quantitative Stability for Linear Perturbations] \label{htyttigui}
    Let $\{F_\delta\}_{\delta \in [0,1)}$ be an admissible $R(\delta)$-perturbation with $R(\delta) = K_5 \delta$. Let $\mu_\delta$ denote the unique $F_\delta$-invariant probability measure in $S^\infty$. Then, there exist constants $D_2 < 0$ and $\delta_1 \in (0, \delta_0)$ such that, for all $\delta \in [0, \delta_1)$, the following stability estimate holds:
    \begin{equation*}
        \|\mu_\delta - \mu_0\|_\infty \leq D_2 \delta \log \delta.
    \end{equation*}
\end{corollary}

\section{Exponential Decay of Correlations and Central Limit Theorem}\label{limit}

In this section, we establish the fundamental limit theorems for the dynamical system $(F, \mu_0)$. Specifically, we demonstrate that the system exhibits \textbf{exponential decay of correlations} for observables within the space $L^1(\mathcal{F}_0)$ (to be introduced below) and the Lipschitz space $\ho(\Sigma)$. Furthermore, we leverage these statistical properties to prove that $(F, \mu_0)$ satisfies the \textbf{Central Limit Theorem} for Lipschitz observables.

\subsection{Exponential Decay of Correlations over constant fiber functions}

Throughout this subsection, the map $F: \Sigma \to \Sigma$ is assumed to satisfy the conditions established in Section \ref{sec1}. We aim to prove the exponential decay of correlations for two classes of observables: functions that are constant along fibers and Lipschitz continuous functions. The following lemma provides the necessary regularity for the fiber-averaged observables, treating the base space as the subshift of finite type $\Sigma_A^+$.

\begin{lemma}\label{uiytirut}
    Let $\varphi \in \ho(\Sigma)$ be a Lipschitz function on the total space $\Sigma$. Then, there exists a disintegration $(\{\mu_{0,\gamma}\}_{\gamma \in \Sigma_A^+}, \phi_1)$ of the invariant measure $\mu_0$ such that the fiber-average function
    \begin{equation*}
    \gamma \longmapsto \int_K \varphi(\gamma, \cdot) \, d\mu_{0,\gamma}
    \end{equation*}
    is a Lipschitz function on the subshift of finite type $\Sigma_A^+$.
\end{lemma}

\begin{proof}
    By Theorem \ref{regg}, the Lipschitz constant of $\mu_0$ is finite. Thus, there exists a disintegration $\omega = (\{ \mu_{0,\gamma} \}_{\gamma \in \Sigma_A^+}, \phi_1)$ of $\mu_0$ such that $|\Gamma_{\mu_0}^{\omega}|_\theta < +\infty$. Since $\mu_0$ is the invariant probability measure, its marginal density $\phi_1$ satisfies $\phi_1 \equiv 1$ $m$-a.e.
    
    Now, consider a Lipschitz function $\varphi \in \ho(\Sigma)$. For each $\gamma \in \Sigma_A^+$, the restriction $\varphi(\gamma, \cdot): K \to \mathbb{R}$ is a Lipschitz function on the fiber $K$ with a Lipschitz constant $L_K(\varphi_\gamma) \leq L_\Sigma(\varphi)$, where $L_\Sigma(\varphi)$ is the global Lipschitz constant.
    
    To establish the Lipschitz regularity on the base subshift, take $\gamma_1, \gamma_2 \in \Sigma_A^+$. We decompose the difference of the fiber averages as follows:
    \begin{align*}
        \left| \int \varphi(\gamma_1, \cdot) \, d\mu_{0,\gamma_1} - \int \varphi(\gamma_2, \cdot) \, d\mu_{0,\gamma_2} \right| 
        &\leq \left| \int \varphi(\gamma_1, \cdot) \, d\mu_{0,\gamma_1} - \int \varphi(\gamma_1, \cdot) \, d\mu_{0,\gamma_2} \right| \\
        &\quad + \left| \int \varphi(\gamma_1, \cdot) \, d\mu_{0,\gamma_2} - \int \varphi(\gamma_2, \cdot) \, d\mu_{0,\gamma_2} \right|.
    \end{align*}
    
    Applying the definition of the Wasserstein metric to the first term and the Lipschitz property of $\varphi$ on $\Sigma$ to the second, we obtain:
    \begin{align*}
        \text{Term 1} &\leq \max\{L_\Sigma(\varphi), \|\varphi\|_\infty\} \|\mu_{0,\gamma_1} - \mu_{0,\gamma_2}\|_W \\
        &\leq \max\{L_\Sigma(\varphi), \|\varphi\|_\infty\} |\Gamma_{\mu_0}^{\omega}|_\theta \, d_\theta(\gamma_1, \gamma_2). \\
        \text{Term 2} &\leq \int_K | \varphi(\gamma_1, \cdot) - \varphi(\gamma_2, \cdot) | \, d\mu_{0,\gamma_2} \\
        &\leq L_\Sigma(\varphi) \, d_\theta(\gamma_1, \gamma_2) \int_K 1 \, d\mu_{0,\gamma_2} = L_\Sigma(\varphi) \, d_\theta(\gamma_1, \gamma_2).
    \end{align*}
    
    Combining these estimates, it follows that the fiber-average function is Lipschitz on $\Sigma_A^+$, completing the proof.
\end{proof}

\begin{athm}\label{çljghhjçh}
    Suppose that the skew-product $F: \Sigma \to \Sigma$, given by $F(\gamma, y) = (\sigma(\gamma), G(\gamma, y))$, satisfies hypotheses (G1) and (G2). Let $\mu_0$ be the unique $F$-invariant probability measure in $S^\infty$. Then, there exists a constant $0 < \tau_2 < 1$ such that, for every \textbf{constant fiber function} $\psi: \Sigma \to \mathbb{R}$ (i.e., $\psi(\gamma, y) = \psi(\gamma)$) satisfying $\psi \in L^1(m)$, and for every Lipschitz observable $\varphi \in \ho(\Sigma)$, we have:
    \begin{equation*}
    \left| \int_\Sigma (\psi \circ F^n) \, \varphi \, d\mu_0 - \int_\Sigma \psi \, d\mu_0 \int_\Sigma \varphi \, d\mu_0 \right| \leq \tau_2^n D(\psi, \varphi), \quad \forall \, n \geq 1,
    \end{equation*}
    where $D(\psi, \varphi) > 0$ is a constant depending on the norms of $\psi$ and $\varphi$.
\end{athm}

\begin{proof}
    Let $\psi: \Sigma \to \mathbb{R}$ be a function that depends only on the base coordinate, $\psi(\gamma, y) = \psi(\gamma)$, with $\psi \in L^1(m)$. Let $\varphi \in \ho(\Sigma)$ be a Lipschitz continuous observable on the total space.
    
    Since $\psi$ is constant along fibers, the composition satisfies $\psi \circ F^n(\gamma, y) = \psi \circ \sigma^n(\gamma)$ for all $n \geq 1$. Moreover, the invariance of the marginal measure $m$ implies:
    \begin{equation*}
    \int_\Sigma \psi(\gamma, y) \, d\mu_0(\gamma, y) = \int_{\Sigma_A^+} \psi(\gamma) \, dm(\gamma).
    \end{equation*}
    
    We now consider the fiber-average function $s: \Sigma_A^+ \to \mathbb{R}$ defined by:
    \begin{equation*}
    s(\gamma) := \int_K \varphi(\gamma, \cdot) \, d\mu_{0, \gamma},
    \end{equation*}
    where $\{\mu_{0, \gamma}\}_{\gamma}$ is the disintegration of $\mu_0$ provided by Lemma \ref{uiytirut}. By that same lemma, $s$ is a Lipschitz function on the base subshift $\Sigma_A^+$.
    
    Using the disintegration of $\mu_0$ and the fact that $\phi_1 \equiv 1$, we can rewrite the correlation integral as:
    \begin{align*}
        \left| \int_\Sigma (\psi \circ F^n) \, \varphi \, d\mu_0 - \int_\Sigma \psi \, d\mu_0 \int_\Sigma \varphi \, d\mu_0 \right| 
        &= \left| \int_{\Sigma_A^+} (\psi \circ \sigma^n)(\gamma) \left[ \int_K \varphi(\gamma, \cdot) \, d\mu_{0, \gamma} \right] dm - \int \psi \, dm \int s \, dm \right| \\
        &= \left| \int_{\Sigma_A^+} (\psi \circ \sigma^n) \, s \, dm - \int_{\Sigma_A^+} \psi \, dm \int_{\Sigma_A^+} s \, dm \right|.
    \end{align*}
    
    Since $m$ is a Markov measure and $s$ is Lipschitz on the subshift $\Sigma_A^+$, the correlation in the last line decays exponentially. This follows from the spectral gap of the transfer operator $\func{P}$ associated with $(\sigma, m)$ (see Proposition \ref{iutryrt}). Therefore, there exists $0 < \tau_2 < 1$ such that:
    \begin{equation*}
    \left| \int_{\Sigma_A^+} (\psi \circ \sigma^n) \, s \, dm - \int_{\Sigma_A^+} \psi \, dm \int_{\Sigma_A^+} s \, dm \right| \leq \tau_2^n D(\psi, \varphi),
    \end{equation*}
    which completes the proof.
\end{proof}

\subsection{Exponential Decay of Correlations over $L^1(\mathcal{F}_0)$}\label{f0}

In this subsection, we introduce the space $L^1(\mathcal{F}_0)$, which serves as a crucial domain for establishing the exponential decay of correlations. This structural result will subsequently facilitate the proof of the Central Limit Theorem for Lipschitz observables.

Let $\mathcal{B} = \mathcal{B}(\Sigma_A^+)$ denote the Borel $\sigma$-algebra of the base subshift $\Sigma_A^+$. Let $\mathcal{G}$ be a countable generator of $\mathcal{B}$, and let $\mathcal{A}$ be the algebra generated by $\mathcal{G}$. We define the collection of cylinders:
\begin{equation*}
\mathcal{A} \times K := \{ A \times K \subset \Sigma : A \in \mathcal{A} \}.
\end{equation*}
It is clear that $\mathcal{A} \times K$ forms an algebra of subsets of the total space $\Sigma$. We denote by $\mathcal{F}_0$ the $\sigma$-algebra generated by $\mathcal{A} \times K$. By construction, $\mathcal{F}_0$ is a sub-$\sigma$-algebra of the Borel $\sigma$-algebra on which the invariant measure $\mu_0$ is defined.

The following proposition characterizes the geometric structure of the sets in $\mathcal{F}_0$, showing that they are composed of entire stable leaves (fibers).

\begin{proposition}\label{nvbvdjkf}
    Let $\gamma \in \mathcal{F}^s$ be a stable leaf and $A \in \mathcal{F}_0$ be a measurable set. If $\gamma \cap A \neq \emptyset$, then $\gamma \subset A$.
\end{proposition}

\begin{proof}
    For a fixed stable leaf $\gamma \in \mathcal{F}^s$, let $\mathcal{C}$ be the collection of sets defined by:
    \begin{equation*}
    \mathcal{C} := \{ A \in \mathcal{F}_0 : A \cap \gamma \neq \emptyset \implies \gamma \subset A \}.
    \end{equation*}
    It is straightforward to verify that $\mathcal{C}$ is a monotone class. Furthermore, $\mathcal{C}$ contains the generating algebra $\mathcal{A} \times K$, since any set $A \times K$ either contains the entire fiber $\{x\} \times K$ or is disjoint from it. By the Monotone Class Theorem, it follows that $\mathcal{C} \supset \mathcal{F}_0$. Since this holds for an arbitrary leaf $\gamma$, the proof is complete.
\end{proof}

\begin{athm}\label{athmc}
    Suppose that $F: \Sigma \to \Sigma$ satisfies hypotheses (G1) and (G2), and let $\mu_0$ be the unique $F$-invariant probability measure in $S^\infty$. Then, there exists a constant $0 < \tau_3 < 1$ such that, for every $\mathcal{F}_0$-measurable function $\psi \in L^1_{\mu_0}(\mathcal{F}_0)$ and every Lipschitz observable $\varphi \in \ho(\Sigma)$, we have:
    \begin{equation*}
    \left| \int_\Sigma (\psi \circ F^n) \, \varphi \, d\mu_0 - \int_\Sigma \psi \, d\mu_0 \int_\Sigma \varphi \, d\mu_0 \right| \leq \tau_3^n D(\psi, \varphi), \quad \forall \, n \geq 1,
    \end{equation*}
    where $D(\psi, \varphi) > 0$ is a constant depending on $\psi$ and $\varphi$.
\end{athm}

\begin{proof}
    The core of the proof lies in showing that any $\mathcal{F}_0$-measurable function $\psi$ is constant along the stable fibers $\gamma \in \mathcal{F}^s$. Let $\psi: \Sigma \to \mathbb{R}$ be such a function. For any $x$ in the range of $\psi$, the pre-image $\psi^{-1}(x)$ is an $\mathcal{F}_0$-measurable set. 
    
    If $\psi^{-1}(x) \neq \emptyset$, there exists a point $z \in \Sigma$ such that $\psi(z) = x$. Let $\gamma_z \in \mathcal{F}^s$ be the stable leaf containing $z$. Since $\gamma_z \cap \psi^{-1}(x) \neq \emptyset$ and $\psi^{-1}(x) \in \mathcal{F}_0$, Proposition \ref{nvbvdjkf} implies that the entire leaf is contained in the pre-image: $\gamma_z \subset \psi^{-1}(x)$. Thus, $\psi$ is constant on each stable leaf $\gamma$.
    
    Furthermore, since $\mu_0$ is a skew-product measure with marginal $m$ on the base $\Sigma_A^+$ and $\phi_1 \equiv 1$, the condition $\psi \in L^1(\mu_0)$ implies that the marginal restriction $\psi(\cdot, y)$ belongs to $L^1(m)$ for $m$-almost every $y \in K$. 
    
    As $\psi$ satisfies the criteria of being a constant fiber function and is integrable, the result follows immediately from the application of Theorem \ref{çljghhjçh}.
\end{proof}

\subsection{Gordin's Theorem}

To establish the Central Limit Theorem for our skew-product system, we employ a classical result from ergodic theory: Gordin's Theorem. This approach approximates the ergodic sums by a martingale difference sequence, provided that the observables satisfy certain summability conditions on their conditional expectations. For a comprehensive proof and further details, we refer the reader to \cite{Stoc}.

Let $(\Sigma, \mathcal{F}, \mu)$ be a probability space. Let $F: \Sigma \to \Sigma$ be a measurable map and $\mathcal{F}_0 \subset \mathcal{F}$ be a sub-$\sigma$-algebra such that the sequence of $\sigma$-algebras defined by $\mathcal{F}_n := F^{-n}(\mathcal{F}_0)$ for $n \in \mathbb{N}$ is non-increasing, i.e., $\mathcal{F}_{n+1} \subset \mathcal{F}_n$. 

A function $\xi: \Sigma \to \mathbb{R}$ is $\mathcal{F}_n$-measurable if and only if there exists an $\mathcal{F}_0$-measurable function $\xi_0$ such that $\xi = \xi_0 \circ F^n$. We denote the space of square-integrable $\mathcal{F}_n$-measurable functions by:
\begin{equation*}
L^2(\mathcal{F}_n) := \{ \xi \in L^2(\mu) : \xi \text{ is } \mathcal{F}_n\text{-measurable} \}.
\end{equation*}
Note that the non-increasing property of the $\sigma$-algebras implies a nested structure of Hilbert spaces: $L^2(\mathcal{F}_{n+1}) \subset L^2(\mathcal{F}_n)$ for all $n \geq 0$. Given any $\phi \in L^2(\mu)$, we denote by $\mathbb{E}(\phi \mid \mathcal{F}_n)$ the $L^2$-orthogonal projection of $\phi$ onto $L^2(\mathcal{F}_n)$.

\begin{theorem}[Gordin] \label{TeoremaDeGordin}
    Let $(\Sigma, \mathcal{F}, \mu)$ be a probability space, $F: \Sigma \to \Sigma$ a measurable map, and $\mu$ an $F$-invariant ergodic probability measure. Let $\mathcal{F}_0 \subset \mathcal{F}$ satisfy the condition that $\mathcal{F}_n := F^{-n}(\mathcal{F}_0)$ is a non-increasing family of $\sigma$-algebras. Let $\phi \in L^2(\mu)$ be an observable such that $\int \phi \, d\mu = 0$. Define the asymptotic variance:
    \begin{equation*}
    \sigma_\phi^2 := \int \phi^2 \, d\mu + 2 \sum_{j=1}^{\infty} \int \phi \cdot (\phi \circ F^j) \, d\mu.
    \end{equation*}
    If the sequence of conditional expectations is summable in $L^2$, i.e.,
    \begin{equation*}
    \sum_{n=0}^\infty \| \mathbb{E}(\phi \mid \mathcal{F}_n) \|_2 < \infty,
    \end{equation*}
    then the series defining $\sigma_\phi^2$ converges absolutely. Furthermore, $\sigma_\phi = 0$ if and only if $\phi$ is a coboundary ($\phi = u \circ F - u$ for some $u \in L^2(\mu)$). If $\sigma_\phi > 0$, then for any Borel set $A \subset \mathbb{R}$ whose boundary has zero Lebesgue measure, we have:
    \begin{equation*}
    \lim_{n \to \infty} \mu \left( x \in \Sigma : \frac{1}{\sqrt{n}} \sum_{j=0}^{n-1} \phi(F^j(x)) \in A \right) = \frac{1}{\sigma_\phi \sqrt{2\pi}} \int_A e^{-\frac{t^2}{2\sigma_\phi^2}} \, dt.
    \end{equation*}
\end{theorem}
\subsection{Application: The Central Limit Theorem}

To apply Gordin's Theorem to our skew-product system, we utilize the specific filtration of $\sigma$-algebras $\{\mathcal{F}_n\}_{n \geq 0}$ introduced in Section \ref{f0}. The following lemma shows that the exponential decay of correlations established in Theorem \ref{athmc} implies the required summability of the conditional expectations.

\begin{lemma}\label{previous}
    For every Lipschitz continuous function $\phi \in \ho(\Sigma)$ satisfying $\int \phi \, d\mu_0 = 0$, there exists a constant $R = R(\phi) > 0$ such that:
    \begin{equation*}
    \|\mathbb{E}(\phi \mid \mathcal{F}_n)\|_2 \leq R \tau_3^n, \quad \text{for all } n \geq 0,
    \end{equation*}
    where $\tau_3 \in (0, 1)$ is the constant from Theorem \ref{athmc}.
\end{lemma}

\begin{proof}
    Recall that by Theorem \ref{athmc}, for any $\psi \in L^1_{\mu_0}(\mathcal{F}_0)$ and $\phi \in \ho(\Sigma)$ with zero mean, we have the correlation bound $|\int (\psi \circ F^n) \phi \, d\mu_0| \leq D(\psi, \phi) \tau_3^n$. 
    
    Using the duality property of $L^2$ and the fact that $\xi \in L^2(\mathcal{F}_n)$ if and only if $\xi = \psi \circ F^n$ for some $\psi \in L^2(\mathcal{F}_0)$, we can express the norm of the conditional expectation as:
    \begin{align*}
        \|\mathbb{E}(\phi \mid \mathcal{F}_n)\|_2 &= \sup \left\{ \left| \int_\Sigma \xi \phi \, d\mu_0 \right| : \xi \in L^2(\mathcal{F}_n), \|\xi\|_2 = 1 \right\} \\
        &= \sup \left\{ \left| \int_\Sigma (\psi \circ F^n) \phi \, d\mu_0 \right| : \psi \in L^2(\mathcal{F}_0), \|\psi\|_2 = 1 \right\}.
    \end{align*}
    Since $\|\psi\|_1 \leq \|\psi\|_2$ by Jensen's inequality, any $\psi$ with $\|\psi\|_2 = 1$ also satisfies $\psi \in L^1_{\mu_0}(\mathcal{F}_0)$. Applying Theorem \ref{athmc}, the term inside the supremum is bounded by $D(\psi, \phi) \tau_3^n$. Since $D(\psi, \phi)$ depends linearly on the norm of $\psi$, the supremum is finite and bounded by $R(\phi) \tau_3^n$.
\end{proof}

We now state the main statistical result for the unperturbed system.

\begin{athm}[\textbf{Central Limit Theorem}] \label{central}
    Suppose that $F: \Sigma \to \Sigma$ satisfies (G1) and (G2), and let $\mu_0$ be the unique $F$-invariant probability measure in $S^\infty$. For any Lipschitz continuous observable $\varphi \in \ho(\Sigma)$, let $\phi = \varphi - \int \varphi \, d\mu_0$. Define the asymptotic variance:
    \begin{equation*}
    \sigma_\varphi^2 := \int_\Sigma \phi^2 \, d\mu_0 + 2 \sum_{j=1}^{\infty} \int_\Sigma \phi \cdot (\phi \circ F^j) \, d\mu_0.
    \end{equation*}
    Then $\sigma_\varphi^2 < \infty$, and $\sigma_\varphi^2 = 0$ if and only if $\phi = u \circ F - u$ for some $u \in L^2(\mu_0)$. Moreover, if $\sigma_\varphi^2 > 0$, then for every interval $A \subset \mathbb{R}$:
    \begin{equation*}
    \lim_{n \to \infty} \mu_0 \left( z \in \Sigma : \frac{1}{\sqrt{n}} \sum_{j=0}^{n-1} \left( \varphi(F^j(z)) - \int \varphi \, d\mu_0 \right) \in A \right) = \frac{1}{\sigma_\varphi \sqrt{2\pi}} \int_A e^{-\frac{t^2}{2\sigma_\varphi^2}} \, dt.
    \end{equation*}
\end{athm}

\begin{proof}
    By Lemma \ref{previous}, the sequence of conditional expectations is exponentially small in $L^2$, which implies that $\sum_{n=0}^\infty \|\mathbb{E}(\phi \mid \mathcal{F}_n)\|_2 < \infty$. Since the measure $\mu_0$ is ergodic (as it is the unique invariant measure in $S^\infty$), all hypotheses of Gordin's Theorem (Theorem \ref{TeoremaDeGordin}) are satisfied. This completes the proof.
\end{proof}


\begin{thebibliography}{99}
\bibitem{ABV}
\newblock  J. F. Alves, C. Bonatti, and M. Viana,
\newblock SRB measures for partially hyperbolic systems whose central direction is mostly expanding,
\newblock \emph{Inventiones mathematicae}, \textbf{140}(2) (2000), 351--398.


\bibitem{PB} [10.1142/3657]
\newblock V. Baladi,
\newblock \emph{Positive Transfer Operators and Decay of Correlations},
\newblock  World Scientific, London, 2000. 

	\bibitem{RRR}[10.1088/1361-6544/aba888]
	\newblock  R. Bilbao, R. Bioni and R. Lucena,
	\newblock Holder Regularity and Exponential Decay of Correlations for a class of Piecewise Partially Hyperbolic Maps,
	\newblock \emph{Nonlinearity}, \textbf{33} (2020).



    
\bibitem{RRRSTAB}
\newblock  R. Bilbao, R. Bioni and R. Lucena,
\newblock Quantitative statistical stability for the equilibrium states of piecewise partially hyperbolic maps,
\newblock \emph{Discrete and Continuous Dynamical Systems}, \textbf{44}(3) (2024).\url{https://doi.org/10.3934/dcds.2023129}


\bibitem{BM}
	\newblock O. Butterley and I. Melbourne,
	\newblock Disintegration of Invariant
	Measures for Hyperbolic Skew Products,
	\newblock preprint, arXiv{1503.04319}. 

\bibitem{RR}
\newblock  R. A. Bilbao, R. Lucena,
\newblock Thermodynamic formalism for discontinuous maps and statistical properties for their equilibrium states, 
\newblock \emph{Chaos, Solitons \& Fractals}, \textbf{207} (2026).

\bibitem{GP} [10.1017/S0143385709000856]
\newblock S. Galatolo and M. J. Pacifico, 
	\newblock Lorenz-like flows: Exponential decay of correlations for the Poincaré map, logarithm law, quantitative recurrence,
	\newblock \emph{Ergodic Theory and Dynamical Systems}, \textbf{30} (2010), 1703--1737.


\bibitem{GLu} [10.3934/dcds.2020079]
	\newblock S. Galatolo and R. Lucena, 
	\newblock Spectral Gap and quantitative statistical stability for systems with contracting fibers and Lorenz like maps,
	\newblock \emph{Discrete and Continuous Dynamical Systems}, \textbf{40}(3) (2020). 


	\bibitem{Hc} [10.1512/iumj.1981.30.30055]
	\newblock J. E. Hutchinson, 
	\newblock Fractals and self-similarity,
	\newblock \emph{Indiana Univ. Math. J.}, \textbf{30} (1981), 713--747.

\bibitem{KL} 
\newblock G. Keller and C. Liverani, 
\newblock Stability of the spectrum for transfer operators,
\newblock \emph{Annali della Scuola Normale Superiore di Pisa. Classe di Scienze. Serie IV}, \textbf{28} (1999), 141--152.



\bibitem{ben} [10.1017/etds.2020.22]
	\newblock B. R. Kloeckner, 
	\newblock Extensions with shrinking fibers,
	\newblock \emph{Ergodic Theory and Dynamical Systems}, (2020), 1-40.


\bibitem{DR}
\newblock  D. Lima, R. Lucena,
\newblock Lipschitz regularity of the invariant measure of random dynamical systems,
\newblock \emph{J. Fixed Point Theory Appl.}, \textbf{27}(44) (2025).\url{https://doi.org/10.1007/s11784-025-01196-1}


	\bibitem{L}
	\newblock R. Lucena,
	\newblock  \emph{Spectral Gap for Contracting Fiber Systems and Applications},
	\newblock  Ph.D thesis, Universidade Federal do Rio de Janeiro in Brazil, 2015.
	

	\bibitem{Kva}[10.1017/CBO9781316422601] 
	\newblock K. Oliveira and M. Viana,
	\newblock \emph{Foundations of Ergodic Theory},
	\newblock  Cambridge University Press, February 2016. 
    \url{https://doi.org/10.1017/CBO9781316422601}



\bibitem{RE} 
\newblock J. Rousseau-Egele, 
\newblock Un Theoreme de la Limite Locale Pour une Classe de Transformations Dilatantes et Monotones par Morceaux, \newblock \emph{The Annals of Probability}, \textbf{11} (1983), 772--788.


\bibitem{Stoc}
\newblock M.Viana,
\newblock Stochastic dynamics of deterministic systems,
\newblock {\em Col\'oquio Brasileiro de Matem\'atica}, 
1997.
    
\end{thebibliography}
\end{document}